\documentclass[11pt]{amsart}
\usepackage[T1]{fontenc}
\usepackage{amssymb, amsthm, amsmath}
\usepackage[dvipdf]{graphicx}%
\usepackage{hyperref}

\newtheorem{theorem}{Theorem}[section]

\newtheorem{claim}[theorem]{Claim}

\newtheorem{corollary}[theorem]{Corollary}

\newtheorem{definition}[theorem]{Definition}

\newtheorem{lemma}[theorem]{Lemma}
\newtheorem{fact}[theorem]{Fact}

\newtheorem{proposition}[theorem]{Proposition}

\DeclareMathOperator{\Int}{Int}

\theoremstyle{definition}
\newtheorem{remark}[theorem]{Remark}

\DeclareMathOperator{\acl}{acl}
\DeclareMathOperator{\cl}{cl}\DeclareMathOperator{\dcl}{dcl}
\DeclareMathOperator{\tp}{tp} 
\DeclareMathOperator{\Th}{Th}

\long\def\symbolfootnote[#1]#2{\begingroup%
\def\thefootnote{\fnsymbol{footnote}}\footnote[#1]{#2}\endgroup}

\def\Ind#1#2{#1\setbox0=\hbox{$#1x$}\kern\wd0\hbox to 0pt{\hss$#1\mid$\hss}
\lower.9\ht0\hbox to 0pt{\hss$#1\smile$\hss}\kern\wd0}
\def\ind{\mathop{\mathpalette\Ind{}}}
\def\Notind#1#2{#1\setbox0=\hbox{$#1x$}\kern\wd0\hbox to 0pt{\mathchardef
\nn=12854\hss$#1\nn$\kern1.4\wd0\hss}\hbox to
0pt{\hss$#1\mid$\hss}\lower.9\ht0 \hbox to
0pt{\hss$#1\smile$\hss}\kern\wd0}
\def\nind{\mathop{\mathpalette\Notind{}}}

\def\nthind{\mathop{\mathpalette\Notind{}}^{\text{\th}} }
\def\uth{\text{U}^{\text{\th}} }
\def\ur{\text{U}}
\def\tho{\text{\th}}

\DeclareMathOperator{\Tho}{\tho}

\def\po{\leq_p}

\def\poo{\leq_p^0}

\def\pot{\leq_t}

\def\mtp{\tp_{\CM}}

\def\mdim{\dim_{\CM}}
\def\&{\wedge}

\def\0{\emptyset}

\def\and{\wedge}
\def\Nn{\mathbb{N}}
\def\CN{\mathcal{N}}
\def\CC{\mathcal{C}}
\def\CM{\mathcal{M}}
\def\CF{\mathcal F}

\def\cupI{\bigcup_i I_i}

\begin{document}

\title{Unstable structures definable in o-minimal theories}


\author[A. Hasson]{Assaf Hasson$^*$}
\thanks{$^*$Supported by the EPSRC grant no. EP C52800X 1}
\address{University of Oxford, Mathematical Institute, 24-29 St Giles', Oxford, OX1 3LB, UK}
\email{hasson@maths.ox.ac.uk }

\date{\today}

\author[A. Onshuus]{Alf Onshuus}

\address{Universidad de los Andes,
Departemento de Matem\'aticas, Cra. 1 No 18A-10, Bogot\'{a},
Colombia} \email{onshuus@gmail.com}

\begin{abstract}
Let $\CM$ be an o-minimal structure with elimination of
imaginaries, $\CN$ an unstable structure definable in $\CM$. Then
there exists $X$, definable in $\CN^{eq}$, such that $X$ with all
the structure induced from $\CN$ is o-minimal. In particular $X$
is linearly ordered. \\ As part of the proof we show:
\texttt{Theorem 1:} If $\dim_{\CM}N = 1$ then any $p\in S_1(N)$ is
either strongly stable or finite by o-minimal. \texttt{Theorem 2:}
If $N$ is $\CN$-minimal then $\dim_{\CM}N=1$.
\end{abstract}

\maketitle

\section{Background}

\subsection{Introduction}

Various results in and around o-minimality give hope that some
classification of theories interpretable in o-minimal structures
may exist. As a first step in that direction Y. Peterzil suggested
to examine whether theories interpretable in o-minimal structures
admitted some sort of analysis in terms of (\th-)minimal types
satisfying Zilber's trichotomy.

Zilber's idea of classifying the combinatorial geometries
associated with minimal types as trivial, linear or ``field like''
is an important source of inspiration for model theoretic
research. Although not always possible, the search for such a
classification has, in many cases, resulted in a deeper
understanding of the fine structure of the theories in question. The
results of \cite{peterzil-starchenko} and \cite{hrushovski-zilber}
leave room for the hope that such a trichotomy may hold for minimal types in
theories interpretable in o-minimal structures.

Naturally, any such classification will have to comprise Zilber's
trichotomy for minimal stable structures in o-minimal theories, a
prospect which seems, at the moment, out of reach. The unstable
case, however, seems quite accessible relying on the Trichotomy
Theorem for o-minimal structures of \cite{peterzil-starchenko} and
Shelah's early analysis of dependent theories.

To tackle the unstable case, Peterzil's suggestion was to
prove first that any such structure interprets an o-minimal set.
In this paper we carry this out, obtaining somewhat sharper
results.

\begin{definition}
$\left.\right.$

\begin{enumerate}
 \item Let $\CN$ be any structure and $X\subseteq N^n$ definable. Say
that \emph{$X$ is finite by o-minimal} if there is a definable
equivalence relation $E$ with finite classes and domain $X$ and a
definable linear order $<$ on $X/E$ such that $(X/E, <)$, with all
the induced structure from $\CN$,  is o-minimal.

\item \noindent A
type $p$ over $A\subset X$ is finite by o-minimal if it is \th-minimal (i.e. has
$\uth$-rank 1, see Definition \ref{thorking}) and such that there
is a non algebraic extension of $p$ containing a finite by
o-minimal formula.
\end{enumerate}
\end{definition}

We can now state our main result:

\begin{theorem}\label{mainintro}
Let $\CN$ be definable in an o-minimal structure. Then either
there is an $\CN$-definable subset of $N$ which is finite by
o-minimal or there is an $\CN$-definable subset of $N$ which is
stable.

Moreover, if $\CN$ is $\kappa$-saturated and $|N_0|<\kappa$ for
some $N_0\subset N$ then every type $p\in S_1^{\CN}(N_0)$ can be
extended to a non algebraic type $q\in S_1^{\CN}(N)$ which is either
finite by o-minimal or strongly stable.
\end{theorem}

The notion of strongly stable types requires explanation:

\begin{definition}
Let $T$ be a dependent theory, $\CN\models T$. A type $p\in S(N)$
is \emph{strongly stable} if there are no $p'\supseteq p$ and
formula $\phi(x,y)$ defining a quasi order with infinite chains in
$p'$.
\end{definition}

Strongly stable types are stable according to Shelah's definition
of stable types in a dependent theory (see \S 1 of
\cite{shelah715}), but the definitions are easily seen not to be
equivalent (see the example concluding Section 2). Types which are not strongly stable will be called
\emph{weakly unstable}.

\medskip 

On the global level we can strengthen the result of Theorem \ref{mainintro}:

\begin{theorem}\label{interpreting}
Let $\CM:=(M,<,\dots)$ be an o-minimal structure with a  dense underlying order and 
elimination of imaginaries. It $\CN$ is unstable, 
interpretable in $\CM$, then $\CN$ interprets an o-minimal structure.
\end{theorem}

These results, together with the more accurate local statement of
the 1-dimensional case (Theorem \ref{theorem}) and the reduction
to it (Section 4, Claim \ref{ominimal}) give rise to the hope that
a structural analysis of types in terms of o-minimal and minimal
stable types could be achieved. An obvious obstacle on the way of such a program is the fact that Theorem \ref{interpreting} does not have an immediate local analogue (strengthening the second part of Theorem \ref{mainintro}), as shows an easy example in Section 3. It seems that, in order to formulate (and prove) a correct analogue of that theorem, machinery such as theories of domination, analysability, stable domination and weight may have to be developed. It is not quite clear to us what is the right context for such a project. Recent work of Shelah in \cite{Sh900} 
suggests some directions (in the significantly more general setting
of dependent theories) which may be of relevance to the present
project.

\bigskip
\noindent The structure of this paper is as follows. In Section
\ref{dimension1} we show that given any 1-dimensional partially
ordered set $(N,\po)$ definable in an o-minimal structure $\CM$
there is a $\po$-definable finite by o-minimal $X\subseteq N$.

The strategy of the proof of Theorem \ref{interpreting} will be to
inductively reduce the problem to the 1-dimensional case. In
Section \ref{dimdown} we perform the first part of the induction
step, showing that if $\CN$ is any structure definable in an
o-minimal $\CM$, then $N$ is $\CN$-minimal only if it is either
strongly minimal or 1-dimensional. In other words if $\CN$ is
unstable and $\dim N >1$ there exists an $\CN$-definable
$X\subseteq N$ with $\dim X < \dim N$.

The induction is completed in section \ref{interpretingsec}, where
Theorem \ref{interpreting} is proved using the machinery of
\th-forking. The proof consists of showing that, assuming
elimination of imaginaries in the ambient o-minimal structure, we
can actually find a lower dimensional unstable set.

\subsection{Dependent theories (or theories with NIP)}

We assume the reader has certain familiarity with basic notions of
model theory, o-minimality and geometric structures. We list some known facts which will be used repeatedly
throughout the paper and refer to \cite{vandendries} and \cite{pillaybook} for the necessary background in o-minimality and stability respectively. 

\medskip

\noindent We begin with:

\begin{definition}
Let $T$ be any theory and $\mathcal C\models T$ a monster model.

\begin{enumerate}
\item a formula $\phi(x,y)$ has the \emph{order property} if there
are indiscernible sequences $\langle a_i\rangle_{i\in \omega}$ and
$\langle b_i\rangle_{i\in \omega}$ such that $\mathcal C\models
\phi(a_i, b_j)$ if and only if $i<j$.

\item A formula $\phi(x,y)$ has the \emph{strict order property}
if there is an indiscernible sequences $\langle a_i\rangle_{i\in
\omega}$ such that $\mathcal C\models \exists y\ \phi(y,
a_i)\wedge \neg \phi(y,a_j)$ if and only if $i<j$.

\item A formula $\phi(x,y)$ has the \emph{independence property}
if there is an indiscernible sequences $\langle a_i\rangle_{i\in
\omega}$ such that for any finite disjoint sets $I$ and $J$ there
is some $c$ such that $\mathcal C\models \phi(c, a_i)$ for any
$i\in I$ and $\mathcal{C} \models \neg \phi(y,a_j)$ whenever $j\in
J$.

\item A theory $T$ is \emph{dependent}(equivalently, does not have
the independence property, or has the non-independence property)
if no formula has the independence property.

\item A definable set $\theta(x)$ is \emph{stable} if there is no
$\phi(x,y)\in \CC$ such that $\phi(x,y)\land \theta(x)$ has the
order property.

\end{enumerate}
\end{definition}

The following theorem is a rehash of results from \cite{shelah715}
and  \cite{shelahbook} presented more conveniently for our needs
in \cite{onshuus-peterzil}.

\begin{theorem}\label{stable}
Let $X=X(\mathcal C)$ be a set interpretable in a dependent
theory. The following are equivalent:

\begin{itemize}
\item $X:=\theta(x)$ is an unstable set.

\item There exists a formula $\phi(x,y)$ such that $\phi(x,y)\wedge
\theta(x)$ has the order property.

\item There are sequences $\langle a_i\rangle_{i\in \omega}$ and
$\langle b_i\rangle_{i\in \omega}$ such that $a_i\in X(\mathcal
C)$ for all $i$, $b_j\in \mathcal C$ for all $j$ and $\mathcal
C\models \phi(a_i, b_j)$ if and only if $i<j$.

\item There is a $\mathcal C$-definable partial order on
$X(\mathcal C)$ with infinite chains.
\end{itemize}

\end{theorem}

Given a dependent structure $\CN$ we will say that an $\CN$-definable
set $X$ is \emph{stable} if it satisfies any of the above
conditions.

\bigskip

As an easy corollary of Theorem \ref{stable} we have the following
theorem which is the technical basis of this paper.

\begin{theorem}[Shelah]\label{shelah}
Let $\CN$ be any structure interpretable in an o-minimal theory
$T$. Then for any $\CN$-definable unstable $X\subseteq N^k$ there is an
$\CN$-definable partial quasi order $\po$ on $X$ with infinite
chains.
\end{theorem}

\begin{proof}
Since $\CN$ is interpretable in an o-minimal theory it does not
have the independence property, and neither does $X$ with all the
induced structure. The result now follows from Theorem \ref{stable}. 
\end{proof}

We conclude with some notational conventions that will be kept
throughout the paper. $\CM:=(M,<,\dots)$ will always denote an
o-minimal structure with a dense underlying order. $\CN$ will be a
structure definable in $\CM$ (in other words, given a structure
$\CN$ definable in some o-minimal structure, we fix such a
structure $\CM$ and an interpretation of $\CN$ therein).
Given any $\CN$-definable $X\subseteq N$ the dimension of $X$,
denoted $\dim(X)$, is the o-minimal dimension of $X$ as an $\CM$-definable set.

\section{The 1-dimensional case}\label{dimension1}

\begin{theorem}\label{theorem}
Let $\mathcal M :=(M,<,\dots)$ be a dense o-minimal structure,
$X\subset M$ an infinite definable set and $\po$ an
$\CM$-definable order with infinite chains on $X$. Let $p\in
S_1(M)$ be a type extending $x\in X$ with infinite
$\lneq_p$-chains. Then for any $e\models p$ there are
$\po$-definable infinite $X' \subseteq X$ with $e\in X'$ generic
and a linear order $<'$ on $X'$ such that $(X',<')\equiv (X',<)$.
\end{theorem}

Some conventions regarding terminology are in place. As in the
statement of the theorem, a set definable in $(X,\po)$ will be
called $\po$-definable (or $X$-definable). The term ``definable
set'' will always refer to $\CM$-definable sets. All orders will
be partial, unless explicitly stated. Hence a quasi order is a
transitive binary relation $\leq$. A quasi order is said to have
infinite chains if the corresponding order (obtained after
quotienting by the equivalence relation $a\leq b \and b \leq a$)
does. Throughout the text all (quasi) orders will be assumed to
have infinite chains. We will also assume, without loss of
generality, that $\CM$ is saturated enough (so that every
$\CM$-definable set has a generic point in $\CM$).

For $x\in X$ it will be convenient to denote $G(x):= \{y\mid x\po
y\}$ and $L(x):=\{y\mid y\po x\}$.

The first part of the proof is to decompose $X$ into well behaved
cells and redefine the partial order to obtain topologically nice
upper and lower cones. Getting the right decomposition of $X$ and
the right partial order to work with are the
main parts of the proof.

\subsection{                                  Taming $(X,\po)$}

By o-minimality, for every $a\in X$, we can write $G(a) =
\bigcup_{i=1}^k I_i$ where each $I_i$ is a definably connected
component of $G(a)$ and $I_i \le I_j \iff i\le j$.  Let $f_j^1(a)$
and $f_j^2(a)$ be the left and right endpoints of $I_j(a)$ (which
coincide if $I_j(a)$ is a point). By allowing empty intervals, we
may assume that $k$ does not depend on $a$ (since the number of connected definably components of $G(a)$  is uniformly
bounded).

Let $X_1, \dots , X_n$ be a decomposition of $X$ such that all the
$f^i_j$ are continuous on each $X_s$ and such that each $f^i_j$ is
either non-increasing or non-decreasing in $X_k$ for all $k$.
With this decomposition of $X$ we redefine the intervals $I_j(a)$
to make sure that $I_j(a)$ is entirely contained in some $X_i$ for
all $a$ and $i$. This can be done as follows: for each $a\in X$
and $i\le k$ define $I_j^i(a) = I_j(a)\cap X_i$. The functions
that define our new intervals will be (weakly) monotone and continuous.

In order to formalise this (and since we will repeat the same
process over and over again in this section), we need the
following definitions.

\begin{definition}\label{MonRep}
Let $\CM$ be an o-minimal structure, $X\subseteq M$ an infinite
definable subset and $\po$ an $\CM$-definable relation inducing a
partial order on $X$. Denote $\CM_{\infty}$ the natural expansion
of $\CM$ to $M\cup \{\pm \infty\}$. A monotone representation of
$(X,\po)$ is a decomposition $X = \bigcup_{i=1}^l X_i$ into
disjoint $\CM$-definable intervals and points, and two finite
collections of definable functions
\[
 \mathcal F^G: \{f_j^i\mid X\to \cl_{\CM_{\infty}}(X)\left| \right. 1\le j \le n, i\in \{1,2\}\}
\]
 and
\[
\CF^L : \{h_j^{i}\mid X\to \cl_{\CM_\infty}(X)\left| \right. 1\le
j \le n, i\in \{1,2\}\}
\]
with
\[ G(x) = \bigcup_{f_j^i\in \CF^G} (f_j^1(x), f_j^2(x))\] and

\[ L(x) = \bigcup_{h_j^{i}\in \CF^L} (h_j^{1}(x), h_j^{2}(x))\] for all $x\in
X$ and such that:

\begin{enumerate}
   \item For every $l\le k, j\le n$ and $i \in \{1,2\}$ the function $f_j^i|_{X_k}$ is continuous and weakly monotone.
    \item For every $1\le j < n$ we have $f_j^1 \le f_j^2 \le f_{j+1}^1$.
    \item For all $x\in X$ and all $j$ there exists $r$ such that $(f_j^1(x), f_j^2(x))\subseteq X_r$.
\end{enumerate} and analogous conditions (1'), (2') and (3') for  $\CF^L$.
\end{definition}

\begin{claim}\label{monrep}
Let $\CM$ be an o-minimal structure. Then given any 1-dimensional
$\CM$-definable partial order $(X,\po)$ and any decomposition
$\{X_i\}$ of $X$ there is a monotone representation
$\{X_i',\CF^G,\CF^L\}$ of $(X,\po)$ such that $\{X_i'\}$ refines
$\{X_i\}$.
\end{claim}

\begin{proof}
We start with the given decomposition $\{X_i\}$ of $X$ and we
refine it as described before Definition \ref{MonRep} to a
decomposition $\{X_i'\}$ so that all the functions $f(x)$ defining
the endpoints of the intervals in both $G(x)$ and $L(x)$ are
either non-increasing or non-decreasing in $X_k'$ for all $k$.

We now define the functions $\{f^i_j\}$ which will define the
endpoints of the ``truncated'' intervals (so that each interval
composing $G(a)$ is entirely contained in a single cell $X_i$).

Let $i_1(a) := \min \{i| I_i(a) \neq \0\}$ and $i_2 := \max \{i|
I_i(a) \neq \0\}$ and define $f_{j,i_1}^1(a) = f_j^1(a)$ and
$f_{j, i_2}^2(a) = f_j^2(a)$.

For $i_1 < i < i_2$ define \[f_{j,i}^1(a) = \max\{f_j^1(a), \inf
X_i'\}\] and \[f_{j,i}^2(a) = \min\{f_j^2(a), \sup X_i\};\] let
\[\CF^G:=\{f_{j,i}^1\}_{i,j}\cup \{f_{j,i}^2\}_{i,j}.\]

Define $\CF^L$ in a similar way. We leave it as a simple exercise
to verify that all the resulting functions are continuous and
weakly monotone on each $X_i'$ so that the collections $\CF^G,
\CF^L$ satisfy Definition \ref{MonRep} with respect to the
decomposition $\{X_i'\}$, which finishes the proof.

\end{proof}

\begin{remark}
To avoid unpleasant trivial cases, after fixing a monotone
representation, we will throw away all the $X_i$ in the
representation consisting of a single point. Replacing $X$ with
$\bigcup \{X_i\mid X_i \mbox{\, is not a point}\}$ we may assume
that the monotone representation we are working with consists of
open intervals only.

To simplify the notation, when referring to a monotone representation,
we will only mention the decomposition of $X$ and use $\CF$ to
denote $\CF^G\cup \CF^L$ whenever no ambiguity can arise.
\end{remark}

\noindent From now on we fix a monotone representation
$\{X_1,\dots X_k, \CF\}$ of $(X,\po)$. To keep the
exposition cleaner, although we will repeatedly refine it (as
explained in Claim \ref{monrep}), we will not change the notation
for the representation. Our first task is to smoothen up $\po$, in
order to make it easier to handle.

\medskip

\noindent By o-minimality the relation

\[ a\poo b \Leftrightarrow \dim(G(b)\setminus G(a))=0.\]
is $\po$-definable. Thus

\begin{lemma}\label{closed G}
$\left.\right.$
\begin{enumerate}

\item $\poo$ is a quasi order refining $\po$ and every definable
$Y\subseteq X$ with infinite $\po$-chains contains infinite
$\poo$-chains.

\item For every $a\in X$ let $G^0(a):=\{x\mid a\poo x\}$ and
$L^0(a):=\{x\mid x\poo a\}$. Then the sets  $G^0(a)\cap X_i$ and
$L^0(a)\cap X_i$ are both relatively closed in $X_i$.

\end{enumerate}
\end{lemma}

\begin{proof}

Since  $a\po b \Rightarrow a\poo b$ and $\poo$ is transitive it is a definable quasi order refining
$\po$. By compactness, if $\po$
had infinite chains then so will $\poo$.

Let $b\in X_i$ and let $a\in X$ be such that $a\not\poo b$, i.e.
$|G(b)\setminus G(a)|$ is infinite, so it contains an interval.
The functions in $\CF$ are continuous at $b$, so there is an
interval $U\subseteq X_i$ with  $b\in U$ such that $|G(x)\setminus
G(a)|$ is infinite (so $a\not\poo x$) for all $x\in U$.
Therefore  $b\notin \{ x\mid  a\po^0 x\}$ implies that $b\notin
\partial \{ x\mid a\po^0 x\}$, proving that $G^0(a)$ is relatively closed in $X_i$.

Now suppose that $a\in X_j$ and that $b\in X$ is such that
$a\not\poo b$ so that $|G(b)\setminus G(a)|$ is infinite, so it
contains an open interval $I$. Since $I\cap G(a)=\emptyset$ the
continuity of the functions in $\mathcal F$ (around $a$ this time)
there is a neighbourhood $V$ of $a$ such that $I\setminus G(a')$ is
infinite for all $a'\in V$; so $a'\not\poo b$ for any such $a'$
implying, as above, that $L(b)$ is relatively closed in $X_j$.
\end{proof}

By Claim \ref{monrep} we can find a
monotone representation of $(X,\poo)$ refining the monotone
representation of $(X,\po)$. So we may assume that
$(X_1,\dots,X_k,\mathcal F)$ is a monotone representation of both
quasi orders. Since Lemma \ref{closed G} is weakened
by the refinement of the monotone representation (there are fewer
interior points), its conclusion will remain valid as we will further refine
$(X_1,\dots,X_k,\mathcal F)$.

It may be worth pointing out that even if $\po$ is an order,
$\poo$ need not be one (i.e. it may be a quasi order). But after
reducing ourselves to a definable subset of $X$ we may assume that
$E(a,b):=a\poo b \land b\poo a$ has finite classes. Thus,
identifying each $E$-class with its smallest element we may assume
the map $\pi: X\to X/E$ is in fact a map from $M$ to $M$. Refining
the above monotone representation further, we may assume that
$\pi$ is continuous on all the cells of the representation. In
particular we may assume that $\poo$ is in fact an order, and we
can work with $\poo$ instead of $\po$, obtaining the following. The fact that the order $\poo$ induces on $X/E$ still satisfies Lemma \ref{closed G} (after possibly removing from $X$ finitely many points) is easy. We obtained:

\begin{fact}\label{assuming closed G} We may assume without loss of generality that $\po$
is an order such that the sets  $G(a)\cap X_i$ and $L(a)\cap X_i$
are both relatively closed in $X_i$ for every $i$ and every $a\in
X$.
\end{fact}

One reason for this additional massaging of our monotone
representation is to obtain:

\begin{corollary}\label{constant}
If $f_j^2(x) = f_{j+1}^1(x)$ for some $j < |\mathcal F /2|$ and
$x\in X_i$ (some i) then both functions are locally constant near
$x$.
\end{corollary}

The proof is immediate from the assumptions of Fact \ref{assuming
closed G} and we leave it as an easy exercise to the reader.
Theorem \ref{theorem} is now proved in two steps. First, we show that
it is enough to find an $\CM$-definable interval where $\po$
agrees with $<$ (the order on $\CM$), and then we proceed to find such an
$\CM$-interval.

\subsection{A special case}

In this subsection we show that if $\po$ agrees with $<$ on some
$\CM$-definable interval $X_0$ then a local version of Theorem
\ref{theorem} follows.

\begin{lemma}\label{from finite to total}
Let $(X,\po)$ be an order definable in an o-minimal structure $\CM
:=(M,<,\dots)$. Assume that $\dim_{\CM} X =1$ and $X=X_0\cup Y$
for some $\CM$-definable $X_0$ such that $\po|_{X_0}$ is a dense
linear order. Then for any $\CM$-generic $e\in X_0$ there exists an
infinite $\po$-definable set $X'$ with $e$ in the interior of $X'$
such that either $\po|_{X'} = \le |_{X'}$ or $\, \ge |_{X'} =\, \po|_{X'}
$.
\end{lemma}

\begin{proof}

Since $X$ is 1-dimensional, we may assume that $X\subseteq M$. Fix
a monotone representation $(X_1,\dots,X_k,\mathcal F)$ of $X$, and
let $e$ be any $\CM$-generic element of $X_0$. In the proof we
will keep shrinking $X_0$, making sure that $e$ is still generic
in the subset of $X_0$ that we keep. Refining our monotone
representation, we may assume that $X_0$ is one of the cells in
the decomposition of $X$.

Reducing $X_0$ if needed and possibly replacing $\po(x,y)$ with
$\po(y,x)$ we may assume, by o-minimality, that $\po$ agrees with
$<$ on $X_0$. Because $\CF$ is a finite collection of functions,
we reduce $X_0$ further to assure that $f^{-1}(c)$ is finite for
all $c\in X_0$ and $f\in \CF$. Note that this can be done without
harming any of the previous requirements.

\medskip

Choose $a<b \in X_0$ generic enough and close to each other such
that $e\in (a,b)_{\po}$. From now on, we will restrict ourselves
to the set $Z_0:=(a,b)_{\po}$. For $x\in Z_0$ denote $u(x) =
\inf\{G(x)\cap X_0\}$ and $m(x):= \sup \{L(x)\cap X_0\}$.

Fix some generic $x\in Z_0$. By assumption (Fact \ref{assuming
closed G}) we know that $m(x)\po x \po u(x)$ and by Fact \ref{assuming closed G} 
$m(x) < u(x)$ for all $x\in ((a,b)_{\po}\setminus X_0)$ (otherwise, $m(x)\po x \po m(x)$
contradicting the assumption that $x$ is an order). We now refine our
monotone representation one last time to assure that $m(x),u(x)$
are continuous and monotone on each cell of the representation;
the assumption that $\dim f^{-1}(c) = 0$ for all $f\in \CF$ and
all $c\in X_0$ assures that both functions will be in fact
strictly monotone.

Let $\mathcal{Z}$ be the collection of infinite $\po$-definable
subsets of $Z_0$ whose interior contains $e$.
The proof proceeds by induction on the
possible cardinalities of the set
\[\left\{i\left| \right. \left\{Z\cap X_i\right\} \text{ is an infinite set.}
\right\}\] for $Z\in \mathcal{Z}$ (and a fixed decomposition
$\{X_i\}$ of $X$ satisfying all the assumptions mentioned up to this point).

Clearly, $|Z\cap X_0|=\infty$ for any $Z\in \mathcal{Z}$ and if
for some such $Z$ we get $\dim(Z\setminus X_0) = 0$ then $\po$ is
a linear order on $Z$ (possibly removing finitely many accidental
points) and the lemma will follow; thus, we may assume by way of
contradiction that any $Z\in \mathcal{Z}$ has infinite
intersection with some $X_i$. Let $Z\in \mathcal Z$ minimise the
number of intervals $X_i$ with which it has an infinite
intersection. For simplicity assume that $Z\cap X_1$ is infinite. 

Let $c\in Z\cap X_1$ be generic. Since $Z_c':=(m(c),u(c))_{\po}$ contains the $X_0$-subinterval $(m(c),u(c))$ it is an
infinite set containing $c$. If $e\notin (m(c),u(c))$ for all
generic $c\in Z\cap X_1$, then by the continuity of $u(x),m(x)$ we know
that either $e>u(c)$ for all but finitely many $c\in Z\cap X_1$ or $e<m(c)$ for all
such $c$. Both cases are analogous so we may assume the latter
holds for all generic $c$. We define $Z':=Z\cap (a,b')_{\po}$ for some $a
< e < b' \le \inf \{m(c)| c\in X_1\cap Z_0\}$; by definition $Z'\cap X_1$ contains no generic points so by continuity $Z'\cap X_1=\emptyset$ and $e$ is in the interior of $Z'$. Since $Z'\subseteq Z$ this contradicts
the choice of $Z$. Therefore we may assume that $e\in (m(c),u(c))$
for some generic $c\in Z\cap X_1$.

We will investigate two cases. Suppose first that $u(x),m(x)$ are
both increasing on $X_1$ (the case they are both
decreasing is similar). This implies that $y\notin
(m(c),u(c))_{\po}$ for all $y\in X_1$ generic over $c$ (for if $y
< c$ then $m(y) < m(c)$ implying - by the definition of $m(y)$ -
that $y\notin G(m(c))$ and if $y > c$ then $u(y) > u(c)$ and
$y\notin L(u(c))$). Since $Z_c:=(m(c),u(c))_{\po}\in \mathcal{Z}$,
this would lead to a contradiction to the choice of $Z$.

The only remaining possibility is that $m(x)$ is increasing and $u(x)$ is decreasing in $X_1$, or
vice versa. Consider $m:= \sup \{m(x)\mid x\in X_1\cap Z\}$. If $m
< e$ then by restricting ourselves (as above) to $(a',b)_{\po}$
for some $m<a'<e$ we get a contradiction to the choice of $Z$ (as
we did there). Otherwise (because $e$ is generic) we
know that $m > e$. By symmetry we may assume that $u < e$ where
$u:=\inf \{u(x)\mid x\in X_1\cap Z\}$ . By continuity and
monotonicity, this means that $G(x)\cap L(x)\cap X_0\neq \0$ for
some $x\in X_1\cap Z$. Since this is impossible, the lemma
follows.

\end{proof}

Note that for every generic $e\in X_0$ the set $X'_e$ we found
satisfying the conclusion of Lemma \ref{from finite to total} was
defined using one of finitely many formulae $\psi_1(x,e),\dots,
\psi_s(x,e)$ (depending, possibly, on parameters independent from $e$, on the monotone representation, but not on
$e$ itself). So Lemma \ref{from finite to total} shows that the formula
$\theta(z)$ given by the disjunction of the formulae ``$\po$
restricted to $\psi_i(x,z)$ is a dense linear order'' is satisfied
by every generic $e\in X_0$, whence it is true of all but possibly
finitely many $e\in X_0$.

\subsection{Reducing to the special case.}
We will now show how to obtain the assumptions of the
previous subsection and apply the result to prove Theorem \ref{theorem}.

\begin{claim}\label{J1J2}
Let $X_k$ be a cell in the representation of $X$. Assume there are
$a,b\in X_k$ such that $a\po b$ and $a<b$ (the case $b < a$ will
have analogous results).
\begin{enumerate}
\item If $f_i^1$ is non-increasing in $X_k$ then so is $f_i^2$ and
if $f_i^2$ is non-decreasing in $X_k$ then so is $f_i^1$; either
of these cases implies that $I_i(a)\cap I_i(b)=\emptyset$. \item
If $f_i^1$ is increasing or constant in $X_k$ and $f_i^2$ is
decreasing or constant, then $I_i(a)\subseteq I_i(b)$.
\end{enumerate}

\end{claim}

\begin{proof}
The first assertions follow from the fact that, since $\po$ is a
partial order, $G(a)\supset G(b)$ and $I_i(a)$ and $I_k(a)$ are
always disjoint intervals by construction, $I_j(b)$ must be
contained entirely in a single interval $I_k(a)$.

Either of the first two conditions imply that $j\neq k$ and the
corresponding conclusions follow. The last assertion is immediate.
\end{proof}

Using the claim, whenever $X_k$ satisfies the the assumptions of
the claim and $x\in X_k$ is generic we can partition the set
indexing $\CF^G$ as follows:

\begin{itemize}
\item $J^k_1$ is the set of integers $i$ such that both $f_i^1$
and $f_i^2$ are non-increasing in $X_k$ or both are constant.

\item $J^k_2$ is the set of integers for which both $f_i^1$ and
$f_i^2$ non-decreasing in $X_k$ but not both are constant.

\item $J^k_3$ is the elements not in $J_k^1$ nor in $J_k^2$.
\end{itemize}

\bigskip

\noindent We are now ready to conclude the proof of the theorem.

\begin{proof}[Proof of Theorem \ref{theorem}]
Let $p$ be any 1-$\CN$-type with infinite chains. We have to show that every $e\models p$ is contained in a finite by o-minimal set.  Because the $X_i$ in
the decomposition of $X$ are $\CM$-definable all realizations
of $p$ are in the same $X_i$; consequently $X_i$ has infinite
$\po$-chains. Therefore, without loss of generality, $X_i$
satisfies the assumptions of Claim \ref{J1J2}; for simplicity
assume $X_i=X_1$.

By the continuity of the functions in $\mathcal F$ and using
Corollary \ref{constant}, for generic $a\in X_1$ there exists
$a_{\epsilon} > a$ with $a_\epsilon\in X_1$ such that for all
$a'\in (a,a_{\epsilon})\cap {X_1}$ we have

\[G(a)\cap G(a') = \bigcup_i I_i(a)\cap I_i(a')\] where $I_i(a) =
(f_i^1(a), f_i^2(a))$.

Moreover, keeping the decomposition of the index set of $\CF^G$ obtained
above, we get that such  $a<a'$ satisfy:

\[G(a)\cap G(a'): = \bigcup_{i\in J_1} (f_i^1(a),f_i^2(x)) \cup
\bigcup_{i\in J_2} (f_i^1(x),f_i^2(a)) \cup \bigcup_{i\in J_3}
(f_i^1(a'),f_i^2(a'))\] and the definition of the $J_i$ implies
that $G(a)\cap G(a')\supseteq G(a)\cap G(a'')$ if and only if and
only if $a'<a''$ for all $a<a'<a''<a_{\epsilon}$ with $a',a''\in
X_1$.

Setting $x <' y$ if and only if $G(a)\cap G(x)\supseteq G(a)\cap
G(y)$ we know that $<'$ is a $\po$-definable quasi order agreeing
with $<$ on $(a,a_{\epsilon})\cap X_1$. So the theorem now follows
from Lemma \ref{from finite to total} and the fact that the linear
order $<'$ can be uniformly defined in a set containing any
generic $e\in X$ such that $e\in X_k$ for some $X_k$ a cell in the
representation of $X$ containing infinite $\po$-chains.
\end{proof}

\subsection{Further remarks}

We conclude with the following observation:

\begin{definition}
A structure $\mathcal N$ is \emph {definable in an o-minimal
structure} $\CM$, if it is interpretable in the real sort of $\CM$
(i.e. the universe of the underlying interpretation is definable).

We will define $N$ to be $k$-dimensional if $k$ is the smallest
integer such that there exists an o-minimal structure $\CM$ and a
definable $S\subseteq M^r$ with $\dim_{\CM}S = k$ such that there
is an interpretation of $\mathcal N$ (coming from $\mathcal M$)
with universe $S$.
\end{definition}

\begin{corollary}\label{last corollary}
Let $\mathcal N$ be an unstable structure definable in an
o-minimal structure, and let $\CM$ witness that $\mathcal N$ is
1-dimensional. Then $\mathcal N$ interprets an o-minimal structure
definable in $\CM$.
\end{corollary}

\begin{proof}
Since $\mathcal N$ is definable in an o-minimal structure it does
not have the independence property. Hence, by theorem \ref{shelah}
there is a definable quasi order with infinite chains whose
universe is (a subset of) $N$. Let $a\sim b \iff a\po b \land b\po
a$. By o-minimality and the fact that $\po$ has infinite chains,
we get that $\sim$ has only finitely many infinite classes (and
infinitely many finite classes). Passing to the definable subset
$N'$ of those elements whose $\sim$-classes are finite the
structure $N/\sim$ is definable in the same o-minimal structure.
The conclusion now follows from Theorem \ref{theorem}.
\end{proof}

\begin{remark}\label{wstable}
It should probably be pointed out that Theorem 1.2 cannot be
strengthened to assure that given a (weakly) unstable type $p$
(i.e. a type that is not strongly stable) there is a definable set
contained in $p$ which is finite by o-minimal. The following
example is due to Kobi Peterzil. Consider the structure $\mathcal
R$ consisting of the (unordered) group $(\mathbb R,+)$ expanded by
a predicate for the interval $[0,1]$. For $0 \le r,t \le  n \in
\Nn$ the formula $\psi_n(r,t):=\forall z (z+r \in [0,n] \to z+t
\in [0,1])$ defines a linear order on $[0,n]$.  Using a simple
quantifier elimination argument it is not hard to verify that no
unbounded linear order is definable in this structure. Now
consider the type $p:=\{\neg (\psi_n(x-r,0) \lor
\psi_n(0,x-r))\mid r\in \mathbb R\}$. $p$ is weakly unstable, but
no formula in $p$ is finite by o-minimal.

Note, however, that $p$ is a stable type according to Shelah (any
Morley sequence in $p$ -- this is well defined because $p$ is
definable -- is an indiscernible set). More specifically, denote
$\le_n$ the order on $[-n,n]$ defined above. Take $\{a_i\}_{i\in
\omega}$ such that $a_0=0$, $a_i-a \not \le_n 0$ for all $a \in
\mathbb R\left\langle a_1,\dots, a_{i-1}\right\rangle $ (the group
generated by $\mathbb R$ and $a_1,\dots, a_{i-1}$). It is not hard
to verify that $\{a_i\}_{i>0}$ is an indiscernible set, and
witnesses the stability of $p$. Note that $p$ does not have
$\ur$-rank, since every non algebraic forking extension thereof is
unstable. It is, however, regular and locally modular.
\end{remark}

\section{Finding a set of smaller dimension.}\label{dimdown}
Keeping in mind the goal of interpreting an o-minimal order in any
unstable structure definable in an o-minimal theory, it is natural
to pursue an inductive argument based on the o-minimal dimension
of the interpretation. Having proved the desired result for the
1-dimensional case, the next step is, given an $n$-dimensional
$\CN$, definable in an o-minimal structure $\CM$, to find an
$\CN$-definable $X\subseteq N$ with $\dim_{\CM}X<n$. This is the
goal of the present section.

Let $\CN,\CM$ be as above with $N\subseteq M^k$. By Shelah's
theorem (Theorem \ref{shelah}) there is an $\CN$-definable quasi
order $\po$ with infinite chains on $N$. For simplicity we may
assume that $\CN = (N,\po)$. We will show that if $N$ is
$\CN$-minimal, i.e. every non-algebraic $\CN$-definable subset of $N$ has
dimension $n$, then $n=1$.

As the next remark shows, if $N$ is $\CN$-minimal it has an
intrinsic notion of dimension. To avoid confusion the use of the term ``dimension'' will
be reserved exclusively for the o-minimal dimension, and ``generic''
will always mean ``$\CM$-generic'' (over the relevant data). To simplify things, we may assume that
$N$ and $\po$ are $\0$-definable in $\CM$. The following appears already in
\cite{PePiSt2}, but we give the simple proof:

\begin{remark}\label{exchange}
Let $\CN$ be a structure definable in an o-minimal structure
$\CM$. If for all $a\in N$ either $\tp(a/A)$ is algebraic or
$\mdim(\tp(a/A))=n$ then $\CN$ is a geometric structure, i.e. the
model theoretic algebraic closure $\acl_{\CN}(\cdot)$ satisfies
the Exchange Property and $\CN$ eliminates the quantifier
$\exists^{\infty}$.
\end{remark}
\begin{proof}
Everything is clear, except exchange. So let $a\in
\acl(Ab)\setminus \acl(A)$. We have to show that $b\in \acl(Aa)$.
If $\dim\tp(b/Aa)<n$ it must be algebraic and the proposition
follows. Otherwise, possibly replacing $b$ we may assume that
$\dim(b/Aa)=n$. Since $a\notin \acl(A)$ we know that $a\nind_A b$
(in the o-minimal sense). Therefore, $b\nind_A a$ implying that
$\dim\mtp(b/Aa) < n$ which contradicts our assumptions.
\end{proof}

\noindent \emph{Throughout this section we will assume that $N$ is
$\CN$-minimal, so in particular every infinite $\CN$-definable
$S\subseteq N$ contains an $\CM$-generic point}. To simplify the
exposition we will assume by way of contradiction that $\dim_{\CM}
N = n > 1$.

\bigskip

Very much like in the 1-dimensional case, the proof goes through
finding an $\CM$-definable set $X$, and an $\CN$-definable
(partial quasi) order $\pot$ such that the restriction of $\pot$
to $X$ is linear. We start with some definitions and results
preparing the ground for what follows. We keep the notation of the
previous section.

\begin{definition}
Let $\CM$ be an o-minimal structure and $S\subseteq X\subseteq
M^k$ definable subsets, $X$ definable over $\0$. $S$ is
\emph{generically closed in $X$} if for every $b\in X$ generic
over $\0$, if $b\in \cl_X (S)$, then $b\in S$.
\end{definition}

Note that the above definition is meaningful only for sets $S$
which are not $\emptyset$ definable. Throughout this section, by
"$S$ is generically closed" we will mean that $S$ is generically
closed in $N$.

\begin{lemma}\label{close section 2}
Suppose $\CN:=(N,\po)$ is an $\CM$-definable structure such that
$N$ is $\CN$-minimal with $\dim_{\CM} N =n$. Let $Z\subseteq N$ be
any $\CN$-definable set and $\pot$ the (partial) quasi order
defined by

\[a\pot b \Leftrightarrow \left|\left\{ y\in Z \left|\right. y\po
a \wedge y\not\po b\right\} \right|<\infty. \]

Then for any $a\in N$ the set $\{x\mid x\pot a\}$ is
generically closed in $N$.
\end{lemma}

\begin{proof}

The lemma is trivial if $Z$ is finite, so we assume it not to be
the case.  Notice that in order to prove the lemma we must show
first that $\pot$ is in fact a quasi order and second that the
cones it defines are generically closed.

To prove it is a quasi order, let $a,b,c$ be elements in $N$ such
that $a\pot b\pot c$. Setting $L(x):=\{y \mid y\po x\}$ we know that both $(L(b)\setminus L(a)) \cap Z$
and $(L(c)\setminus L(b))\cap Z$ are finite.

But

\[L(c)\setminus L(a)\subset \left(L(c)\setminus L(b)\right)\cup
\left(L(b)\setminus L(a)\right)\] so $(L(c)\setminus L(a))\cap Z$ is
finite and by definition $a\pot c$ so $\pot$ is transitive.

To prove that $L_t(a)$ is generically closed, suppose that
$c\not\pot a$ for some $c$ generic over $\0$. By definition there
are infinitely many points in $Z\cap (L(c)\setminus L(a))$. Since
this last set is $\CN$-definable we know by assumption that it has
dimension $n$ so there is an open set $U\subseteq Z\cap L(c)$ such
that $U\cap L(a)=\0$. Choosing $U$ small enough, we may assume
that it is definable over parameters independent over all the
data, so by genericity of $c$ there is a neighbourhood $V$ of $c$
such that $U\subseteq L(x)\cap Z$ for all $x\in V$ whence $x\not
\pot a$ as required.
\end{proof}

\noindent It will be important for applications to note that in the above lemma we do not assume that $Z$ is
$\0$-definable.

\medskip

Out next step is to find an $\CN$-definable order $\pot$ and
an $\CM$-definable line $l$ through $\CN$ (see below) such that $\pot$
restricted to $l$ is a linear order. Our way of obtaining this is
reminiscent of Shelah's proof of Theorem \ref{shelah}.

\bigskip

\noindent The following technical result will be needed:

\begin{proposition}\label{random crap}

Let $\CN:=(N,\po)$ be a partially ordered set with infinite chains
definable in an o-minimal structure $\CM$ and $\dim_{\CM} N =n$.
Assume also that $N$ is $\CN$-minimal. Then:

\begin{enumerate}
\item There is an infinite $\po$-chain $a_0\po \dots \po a_i\po
a_{i+1}\po \dots$ of elements in $N$ such that $a_i$ is generic
for all $i$.

\item There is an $\CM$-generic type $p(x)\in S_1(\0)$ and
an infinite $\po$-chain $a_0\po \dots \po a_i\po a_{i+1}\po \dots$
such that $\models p(a_i)$ for all $i$.

\item There is a type $p(x)\in S_1(\0)$  and elements $a,b\models
p$ such that $b$ is a generic element in $\partial G(a)$.
\end{enumerate}
\end{proposition}

\begin{proof}

$\left.\right.$

\noindent (1). By assumption we have an infinite $\po$-chain so by
compactness we can find a sequence $\langle
x_{(i,j)}\rangle_{(i,j)\in \omega\times \omega}$ where
$x_{(i_1,j_1)}\po x_{(i_2,j_2)}$ if and only if
$(i_1,j_1)<(i_2,j_2)$ in the lexicographic order. Let
$c_i:=x_{(i,0)}$; by definition $(c_i, c_{i+1})_{\po}$ is infinite
for any $i\in \omega$.

By $\CN$-minimality $\dim (c_i, c_{i+1})_{\po} = n$ so there is
some $\CM$-generic $a_i\in (c_i, c_{i+1})_{\po}$. Then $\langle
a_i\rangle$ is an infinite $\po$-chain of $\CM$-generic elements.

\noindent (2). Using (1) and compactness we can find an
arbitrarily long $\po$-chain of $\CM$-generic elements. By Ramsey's Theorem
we can find a $\po$-chain $\langle a_i\rangle$ of $\CM$-generic
elements such that $\tp(a_i/\emptyset)=\tp(a_j/\emptyset)$ for all
$i,j$.

\noindent (3). By (2) there is an $\CM$-generic type $p(x)$ such
that there are infinite $\po$-chains among realizations of $p(x)$.
Let $a\models p(x)$ so that for any $\phi(x)\in p(x)$ both
$G(a)\cap \phi(N)$ and $\neg G(a)\cap \phi(N)$ have dimension $n$.
By \cite{johns} $\dim \partial G(a)\cap \phi(N) = n-1$ for any
$\phi(x)\in p(x)$. By compactness there exists $b\in
\partial G(a) \cap p$  such that $\dim(b/a)=n-1$. The elements
$a,b$ will satisfy the requirements of (3).
\end{proof}

A \emph{line through $N$} is a 1-dimensional ($\CM$-definable)
definably connected subset of $N$. Say that a line $l$ through $N$
is \emph{generic} if any generic $a\in l$ is generic also (over
$\0$) in $N$. So a line $l$ through $N$ is definably homeomorphic
to an interval in $\CM$. Fixing such a homeomorphism, $l$ inherits
an ordering form $\CM$. Throughout this section, we will assume
implicitly that lines come equipped with some such ordering. The
only requirement we will make is that when working with a family
of lines the ordering on all lines is given uniformly.

Let $l$ be a line through $N$. Given $b\in X$ there
are unique maximal closed intervals $I_0, I_1, \dots , I_n$ such
that

\begin{enumerate}
\item The right endpoint of $I_j$ is the left endpoint of
$I_{j+1}$.

\item $l=\bigcup_{i} I_i$

\item Either $I_j=\overline{I_j\cap G(b)}$ or
$I_j=\overline{I_j\setminus G(b)}$.

\end{enumerate}

Whenever these conditions hold we will say that $\cupI$ (or
$\cupI(b)$ if we want to make $b$ explicit) is the partition of
$G(b)\cap l$ (or the partition of $l$ with respect to $b$).

If $\bigcup_{i=1}^r I_i$ is a partition of $G(b)\cap l$ for fixed
$l$ and $b$ let $f_l^b:r\rightarrow 2$ be such that $f_l^b(i)=1$
if $I_j=\overline{I_j\cap G(b)}$ and $f_l^b(i)=0$ otherwise; let

\[s_l(b):=\langle f_l^b(i)\rangle.\] We will say that $s_l(b)$
has a sign change at $j$ if $f(j)=0$ and $f(j+1)=1$ or vice versa.

For a fixed line $l$ say that $\cupI$ is locally constant at $b$
if for every open neighbourhood $U$ of $b$ the set
\[\left\{ y \left|\right. \text { $\cupI$ is a partition of
$G(y)\cup I$} \right\}\] contains an open subset of $U$. We will
say that $\cupI$ is locally constant if it is locally constant at
some $b\in X$.

Finally, given any family $\mathcal{B}$ of pairs $(l,b)$ as above,
we will say that the sequence $s_l(b)$ (originating from the
partition of $l$ with respect to $b$) is \emph{maximal} in
$\mathcal{B}$ if it has a maximum number of sign changes (i.e.
$|s_l(b)|$ is maximal) among all $s_{l'}(b')$ with $(l',b')\in
\mathcal{B}$; we will call $s_{l}(b)$ the \emph{type} of the
partition $G(b)\cap l$.

\begin{claim}\label{finite}
Fix a line $l$. Then $l$ admits only finitely many locally
constant partitions.
\end{claim}

\begin{proof}
Let $A$ be any set such that $l$ is definable over $A$. Suppose
that $\cupI$ is a locally constant partition, and
let $a_1,\dots , a_{r}$ be the left endpoints of $I_1, \dots ,I_r$
respectively. Let $U$ be the set of points such that $\cupI$ is a
partition of $G(b)\cap I$ for all $b\in U$. By hypothesis $U$ has
dimension $n$ so there is some $b\in U$ generic over $A\cup \{a_1,
\dots, a_r\}$. Clearly $a_i\in \dcl(b,A)$ and $b\ind_A a_1,\dots,
a_r$ so by symmetry $a_i\in \dcl(A)$. But the set of $x\in l$ such
that $x$ is an endpoint of an interval in a constant partition of
$l$ is $\CM$-definable. Since it is contained in $\acl(A)$
it must be finite.
\end{proof}

It is not hard to verify that Proposition \ref{random crap}
implies the existence of a generic line through $N$ admitting a
non locally constant partition. Formally, (3) of the proposition
implies that we can find $a\in N$ generic and $b\in \partial G(a)$
generic as such, such that $b/\0$ is generic as well. This
implies, in particular, that $\dim \{a'\mid b\in \partial G(a')\}
< n$. Since $a$ was generic, any line $l$ through $b$ will have
non-constant partitions (witnessed by $a$) and, since $b/\0$ is
generic, if $l$ is chosen so that $b$ is generic on $l$ we get
that this line is generic. Observe, moreover, that for such $a$
there is a neighbourhood $U$ of $a$ such that for no $a'\in U$ is
the partition $a'$ induces on $l$ locally constant. So the set of
$a'$ inducing a non-constant partition on $l$ is $n$-dimensional.

Fix such $l_0$ and $b_0$ and let $L$ be a definable family of
lines through $N$ of which $l_0$ is a generic member. Let
$\mathcal B$ be the subset of $N\times L$ of all $(b',l')$ such
that there exists an open set $V$ containing $b'$ where $G(y)\cap
l'$ is a non constant partition of $l'$ of the same type as
$G(b')\cap l'$ for all $y\in V$. Let $l\in L$ be generic and
$(b,l)\in \mathcal B$ be such that the partition of $l$ with
respect to $b$ is maximal among all elements in the set

\[
\{b'\mid (b',l)\in \mathcal B\}.
\]
By definition, for all $b'\in M$ if the size of the partition
$G(b')\cap l$ is greater than the size of the partition of
$G(b)\cap l$ then either $G(b')\cap l$ is a constant partition or
$b'$ is not generic (over $l$). Specifically,

\[
\dim\{b'\mid \text{ $G(b')\cap l$ non constant and }
|s_l(b')|>|s_l(b)|\} < n.\tag{$\clubsuit$}
\]

\bigskip

\noindent \emph{From now on we fix $b,l$ as above.}

\begin{remark}\label{isolated points}
Let $a$ be a generic point of $l$. If $a\in G(b')$ is isolated in
$G(b')\cap l$ then (by symmetry) $\dim(b'/a) < n$.
\end{remark}

\begin{lemma}
Let $\mathbb D$ be the set of all $d\in N$ such that
$|s_l(d)|> |s_l(b)|$. Then there is an $\CN$-definable set
$\mathbb D'$ such that $\dim \mathbb D'\triangle \mathbb D < n$.
In particular, any infinite $\CN$-definable subset of $N\setminus
\mathbb D'$ intersects $N\setminus \mathbb D$ in an infinite set.
\end{lemma}

\begin{proof}
By Claim \ref{finite} we can find $d_1,\dots, d_k$,
representatives of the locally constant partitions of $l$, such that for
any $y$ if the partition of $G(y)\cap l$ is locally constant and
$|s_l(y)|> |s_l(b)|$ then $G(y)\cap l=G(d_i)\cap l$ for some $i$.

Let $J_0$ be the set of $(i,j)$ such that $f_l^{d_j}(i)=0$ and let
$J_1$ be the set of $(i,j)$ such that $f_l^{d_j}(i)=1$. Let
$c_{ij}\in I_i(d_j)$ be generic points in $l$ for all $(i,j)$.

Let
\[
D_j:=\left\{x \left|\right. \bigwedge_{(i,j)\in J_1} c_{ij}\in
G(x) \wedge \bigwedge_{(i,j)\in J_0} c_{ij}\not\in G(x)\right\}.
\]

\medskip

\noindent If $x\in D_j$ then either $|s_l(b)|<|s_l(x)|$, implying that
$x\in \mathbb D$, or $c_{ij}$ is an isolated point in $G(x)\cap l$
for some $i\in J_0$. By remark \ref{isolated points} this implies
that the latter case can only occur if $\dim x/l < n$.
Conversely, if $|s_l(x)|> |s_l(b)|$ then either $x$ induces on $l$ a
constant partition, in which case $G(x)\cap l=G(d_j)\cap l$
for some $j$ and $x\in D_j$; or $x$ is not generic in $\{x\mid
(x,l)\in \mathcal B\}$ by ($\clubsuit$) above. Setting

\[\mathbb D':=\bigcup_j D_j\] the conclusion of the lemma follows.
\end{proof}

Recall that we fixed some generic $b$ inducing a partition of $l$
maximal among all generic $b'$ such that $(b',l)\in \mathcal B$.
Let $J_1:=\{j\mid f^b_l(j)=1\}$ and $J_0:=\{j\mid f^b_l(j)=0\}$.
Let $a_j$ be generic points in the interior of $I_j(b)$. Define:
\[
 Z_0:= \left\{y\mid \bigwedge_{j\in J_1} a_j\in G(y)\wedge
\bigwedge_{j\in J_0} a_j\not\in G(y)\right\}
\]
and
\[
Z:=\{y\in Z_0\mid y\notin \mathbb D'\}.
\]
Finally, we can define our partial order $\pot$:
\begin{equation*}
x_1 \pot x_2\iff |\{  y\in Z\mid y\po x_1 \wedge y\not\po x_2\}| <
\infty.\tag{$\dagger$}
\end{equation*}
By Lemma \ref{close section 2} $\pot$ is indeed a partial order,
and $\{y\mid y \pot a\}$ is generically closed for all generic
$a\in N$. We now show that $\pot$ has the property we were looking
for, namely that restricted to some line (not surprisingly, $l$)
through $N$ it is linear.

To see this note that, as $G(b)$ does not induce a constant
partition on $l$, there is some endpoint $d$ of one of the
intervals in $G(b)\cap l$ witnessing it. Because $d$ is generic in
$l$ we can find a neighbourhood $V\cap l$ of $d$ such that (without
loss of generality) for all $d_1 < d_2$ ($d_1,d_2\in V\cap l$) if
for some $y$ we have
\[
\bigwedge_{j\in J_1} a_j\in G(y)\wedge \bigwedge_{j\in J_0}
a_j\not\in G(y) \wedge d_1\in G(y) \wedge d_2\not\in G(y)\tag{$*$}
\]
then either $|s_l(y)| > |s_l(b)|$ and $y\in \mathbb{D}$ or $a_j$
is an isolated point in $G(y)$ for some $j\in J_1$. Since
$\dim(\mathbb D \triangle \mathbb D')<n$ by Remark \ref{isolated
points} there can be only finitely many $y$ outside $\mathbb D'$
satisfying ($*$). So restricted to $V\cap l$, $\pot$
is a quasi order agreeing with the natural order on $l$.

It remains to verify that $\pot$ is an order (on $V\cap L$). But
since $b$ is generic over all the data there are neighbourhoods $U$
of $b$ and $W\subseteq l$ of $d$ such that for all $d'\in W$ there
is some $b'\in U$ with $d'$ an endpoint of an interval in the
partition of $l$ induced by $b'$ and $a_i\in G(b')$ if and only if
$a_i\in G(b)$ (see Figure 1 below). So for every $x_1\neq x_2 \in
W$ there are infinitely many $b'$ in $U$ separating them and $x_1
\not\equiv_t x_2$.

\begin{center}
\includegraphics[height=1.7in,width=4in]
{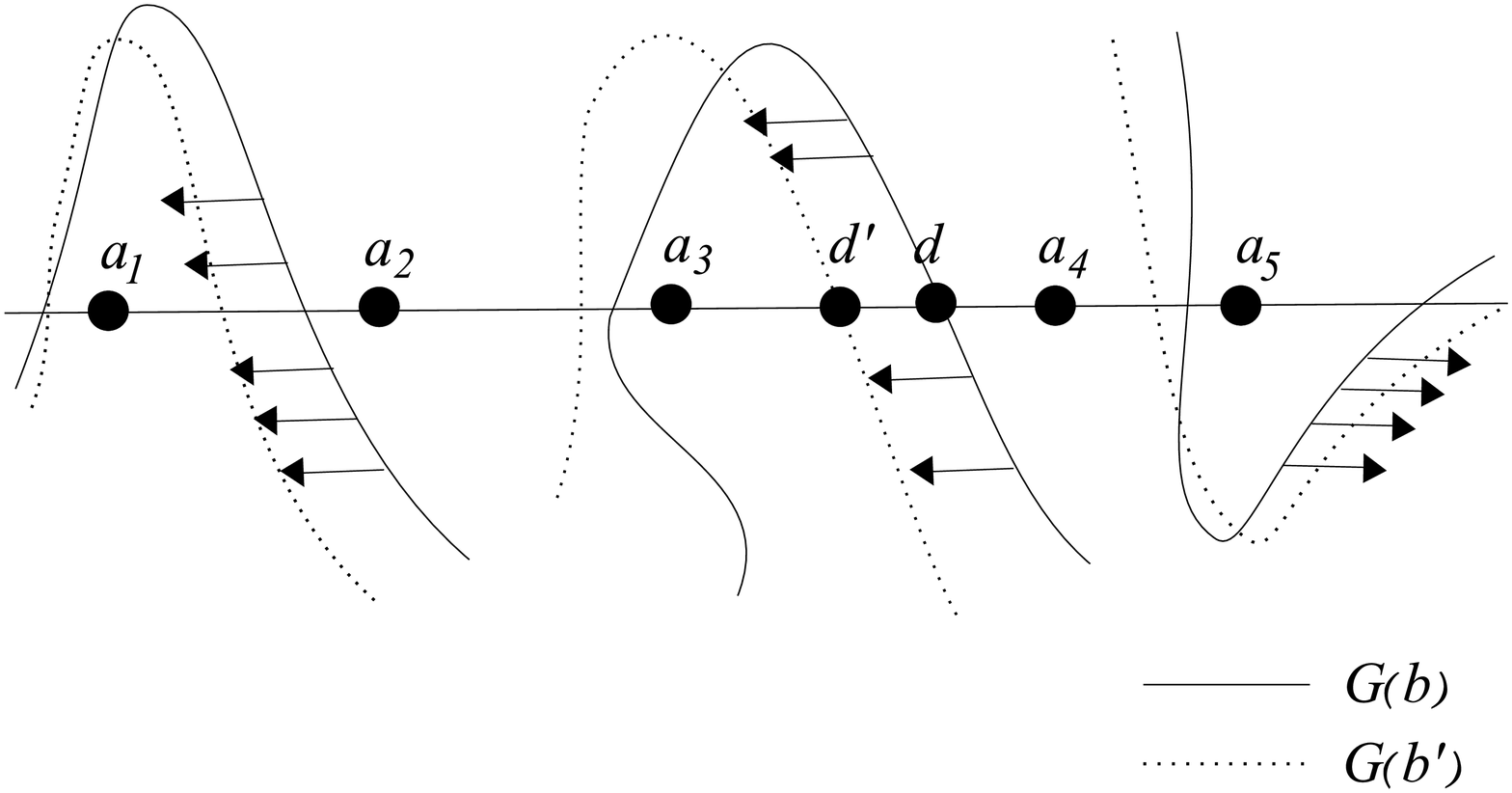}\\
Figure 1: $d'\pot d$
\end{center}

\subsection{Going down}\label{gdown}
We can now prove the main result of
this section:

\begin{theorem}\label{dimension down}
Let $\CN$ be definable (over $\0$) in an o-minimal structure
$\CM$. Assume that $\dim_{\CM} N = n >1$. Then there exists an
infinite $\CN$-definable (over $B$) set $X\subseteq N$ such that
$\dim(X) < n$ and $\ur(\tp(g/\0)) = \infty$ for some $g\in
(X\setminus \acl_{\CM}(B))$
\end{theorem}

\begin{proof}
Assume not. So we can construct an $\CN$-definable (quasi) order $\pot$ linear on some generic line $l$ through $\CN$. By Lemma \ref{close section 2} the lower cones of $\pot$ are generically closed. We fix $\pot$, $l$ and $a,b\in l$ such that $\pot$ is linear on $[a,b]_{\pot}\cap l$. 

We need some notation and easy observations. First, by the genericity of $l$,
any ($\CM$)-generic $c\in l$ is also $\CM$-generic over $\0$. Denote $[a,b]^l:= [a,b]_{\pot} \cap l$ and let

\[
 m(y):=\max_{[a,b]^l} \left\{x\in (a,b)^l \mid x\pot y\right\}
\]
for any $y\in N$ ($m$ need not be globally defined).
With this notation, if $y\in (a,b)_{\pot}$ then $m(y)\in
[a,b]^l$ and since $m(y)\in \cl\{x\in
(a,b)^l \mid x\pot y\}$ Lemma
\ref{close section 2} assures that $m(y)\pot y$ whenever $m(y)$ is generic in
$l$.

\begin{claim}\label{not dimension n}
We may assume that $\dim(m^{-1}(c))<n$ and that $c\pot m^{-1}(c)$
for all $c\in (a,b)^l$.
\end{claim}

\begin{proof}
Since $\dim(N)=n$, necessarily, $\dim(m^{-1}(c))<n$ for all but
finitely many $c$. On the other hand the formula $\exists y
(m(y)=c\wedge c\not\pot y)$ implies that $c$ is not generic in
$l$ so there are only finitely many such.

Thus, by shrinking our interval if necessary, we can find some
$[a',b']^l$ subinterval of $[a,b]^l$ such that $\dim(m^{-1}(c))<n$
for all $c\in [a',b']^l$. Replacing $[a,b]^l$ with $[a',b']^l$ it is easy to verify that  the claim follows.
\end{proof}

\begin{claim}\label{finite fibers}
If $\dim(m^{-1}(x)) < n-1$ for all $x\in
(a,b)^l$ then $\dim( (a,b)_{\pot})<n$.
\end{claim}

\begin{proof}
If $a <_t x <_t b$ then $a\le m(x) \le b$ so

\[(a,b)_{\pot}\subseteq \bigcup_{x\in [a,b]^l} m^{-1}(x).\] 
Therefore $\dim(m^{-1}(x))<n-1$ for all such $x$ implies, by the additivity of o-minimal dimension, that $\dim ((a,b)_{\pot})<n$.
\end{proof}

Because $(a,b)_{\pot}$ has infinite
$\pot$-chains (e.g. the interval in $(a,b)^l$) it must be
unstable; so if $\dim(a,b)_{\pot}<n$ the theorem follows. Hence we may assume that $\dim ((a',b')_{\pot})=n$ for all
$a < a'<b'< b$ in $l$ and that, by Claim \ref{finite fibers},
$\dim(m^{-1}(x))=n-1$ for infinitely many $x\in (a,b)^l$. By
o-minimality of $l$ (and definability of the o-minimal dimension)
there is a subinterval $(a',b')^l\subseteq (a,b)^l$ such that
$\dim(m^{-1}(x))=n-1$ for all $x\in (a',b')^{l}$.

\begin{claim}
Let $a',b'\in l$ be such that $\dim(m^{-1}(x))=n-1$ for all $x\in
(a',b')^l$. Then $\ur(\tp(g/A))=\infty$ for any parameter set $A$ and $g\in (a',b')_{\pot}$ generic over $A$. 
\end{claim}

\begin{proof}
Fix $A$ and $g\in (a',b')_{\pot}$ generic over $A$. Let $\langle g_i\rangle_{i\in \mathbb{Q}}$ be an $\CM$-independent
$\CM$-indiscernible sequence in $\tp_{\CM}(g/A)$. Since $g$ is generic over $A$, in particular, $m(g)$ is generic in $l$, and therefore $m(g_i)\neq m(g_j)$ for all $i\neq j$. Let
$c_i:=m(g_i)$ so the sequence $\langle c_i\rangle_{i\in
\mathbb{Q}}$ is a $\pot$-linearly-ordered indiscernible sequence, so without loss of generality it is $\pot$-increasing. 

By hypothesis $c_i\lneq_t g_j$ if and only if $i<j$. As in the case of dense
linear orders, the formula $(c_i\lneq_t x) \wedge \neg(c_j\lneq_t
x)$ divides over $A$, and is realized by $g_k$
for all $i<k<j$. Since our set is indexed by $\mathbb{Q}$ we get
an infinite dividing sequence which witnesses that
$\ur(\tp(g/A))=\infty$.
\end{proof}

From now on, we will assume $a',b'\in l$ are such that
$\dim(m^{-1}(x))=n-1$ and $m(x)\pot x$ for all $x\in
(a',b')_{\pot}$. Let $g\in (a',b')_{\pot}$ be generic over all the
parameters defining $l$, so by the last claim $\ur(\tp(g/l))=\infty$. Denoting 
$c=m(g)$ we know that $c$ is generic in $(a',b')^l$. The
following lemma will complete the proof of Theorem \ref{dimension
down}.

\begin{lemma}
Let $g,c$ be as above then there exists a set $X$ with $\dim X <
n$ such that either $X$ is  $\CN$-definable over $c$ and $g\in X$
or $X$ contains an infinite $\pot$-chain.
\end{lemma}

\begin{proof}

Since $g$ is generic over $l$ and $\dim(c/l)=1$ the
additivity of o-minimal dimension gives $\dim(\tp(g/c))=n-1$. So
any $c$-definable set containing $g$ is non algebraic.

Notice that by density of $\pot$ on $(a',b')^l$ if $y\in (a',b')^l$ is such
that $c\neq m(y)$ then either $(c,y)_{\pot}$ is empty or it
contains an infinite $\pot$-chain (and infinitely many points) in
$(a,b)^l$. Writing 
\[
X_n(c):= \left\{y\left|\right.
(c,y)_{\pot} \text{ does not contain $\pot$-chains of size larger
than $n$}  \right\}
\]
this implies that either $g\in X_n(c)$ for some $n\in \Nn$, and $X_n(c)$ being $\pot$-definable the lemma
follows, or $g\notin X_n(c)$ for all $n$. 

In the latter case, by saturation of $\CM$, the interval $(c,g)_{\pot}\subseteq m^{-1}(c)$ contains infinite
$\pot$-chains. By assumption, this implies that $\dim (c,g)_{\pot} < n$ and being unstable it satisfies the conclusion of the lemma.
\end{proof}

\noindent This finishes the proof of the theorem.
\end{proof}

\subsection{An alternative proof}
In this subsection we propose a different approach to the proof of
Theorem \ref{dimension down}. We find the proof instructive in the way it
allows us to control local phenomena in reducts of o-minimal
theories. However, being technically more involved, we do not give
all the details. For simplicity and concreteness we discuss the
case $N\subseteq M^2$ and $\dim N = 2$. The interested reader should not find it hard to convince himself (or herself, or others) that the proof extends to the general case, but even the uninterested reader would probably not need much convincing in agreeing that the proof is unpleasant enough as it is, even without the additional technicalities such a generalisation would require. 

To make life a little easier we will prove a slightly weaker version of the theorem. We prove: 

\begin{theorem}
 Let $\CN$ be definable in an o-minimal structure $\CM$. If $\CN$ is unstable and $\dim N =n >1$ then there exists an infinite  $\CN$-definable $X\subseteq N$ such that $\dim X < n$. 
\end{theorem}

The proof starts at the same point where Section \ref{gdown} does,
and we keep the assumptions and notation accumulated up to
that point. In particular, we have a fixed generic line $l$
through $N$ and the $\CN$-definable order $\pot$ obtained above.
Recall that the restriction of $\pot$ to $l$ coincides with the
natural, o-minimal, order on $l$ (induced from some definable
homeomorphism of $l$ with an interval in $M$). Because of the
assumption that $N$ is of full dimension (in $M^2$) we may choose
the line $l$ parallel to one of the axis. In the general case
($N\subseteq M^k$ and $\dim N = n$) we choose an $n$-dimensional
cell $N'\subseteq N$ which is naturally definably homeomorphic to
an open box $B\subseteq M^n$; in that case we choose $l'$ through
$B$ parallel to one of the axis and set $l$ to be the image of
$l'\cap B$ under the inverse homeomorphism.

The assumption that $\pot$ agrees with the order on $M$ on some
line $l$ through $N$ implies that there is an infinite set of
$x\in N$ such that $x$ is not an isolated point in $\cl(L_t(x))$.
The first part of the proof consists in showing that if for some $a$, generic over all the data (including the
parameters required to define $\pot$), $a$ is not isolated in
$L_t(a)$ then the desired result follows. We will then show how
to change $\pot$ to obtain such an $\CM$-generic $a$.

\medskip

For the first part we need the following easy claim:
\begin{claim}\label{finite succesors}
Let $a\in N$ be $\CM$-generic and $c\in \partial L_t(a)$ be
generic as such. If, in addition, $c/\0$ is generic then $\{z\mid c\po z
\po a\}$ is finite and $\{b\in L_t(a)\mid |(b,a)_{\pot}|<\infty\}$
is $n$-dimensional.
\end{claim}
\begin{proof}
As we have already showed we may assume that $\pot$ is generically
closed (i.e., that $L_t(x)$ is generically closed for all generic
$x$). Assume that $a,c$ are as above and that $c/\0$ is generic,
therefore $c\nind^{\CM}a$. Hence, $\dim\{a'\mid c\in \partial
L_t(a')\} <n$ and $c\in L_t(a)$.

By construction, if $c\pot d\pot a$ then $c\in \partial L_t(d)$;
so for any such $d$ we have $c\nind^\CM d$ which implies
$\dim(\{z\mid c\pot z\pot a\})<n$. But $\{z\mid c\pot z\pot a\}$
is $\CN$-definable so by assumption it must be finite.

Since $b$ was generic in $\partial L_t(a)$ we get that $\{b\in
L_t(a)\mid |(b,a)_{\pot}|<\infty\}$ is infinite, so it must be
$n$-dimensional.
\end{proof}

On the other hand, if $a\in N$ is generic and $a$ is not isolated
in $L_t(a)$ then for every $b\in L_t(a)$ generic over $a$ it must
be that $a\in \Int G_t(b)$. Thus the $\pot$-interval
$(b,a)_{\pot}$ is infinite, contradicting the previous claim.

\bigskip

We can now turn to showing that there is some partial order with
infinite chains $\pot$ and a generic $a$ such that $a$ is not an
isolated point in $L_t(a)$. As in the previous proof, we fix $b$
generic inducing a partition on $l$ maximal among all generic $b'$
such that $(b',l)\in \mathcal B$. We also fix generic $a_j\in
I_j(b)$. Denote $A$ the collection of the $a_j$ and
\[
R_b(x,A): = \left(\bigwedge_{a_j\in A} (a_j\in G(x) \iff a_j\in G(b))\right ).
\] 
Say that $x$ is good for $A$ (with respect to $b$)
if $\models R_b(x,A)$. Observe that, since $b$ was generic and the
$a_i$ were chosen independent of $b$ (over all the data), if $x$
is good for $A$ with respect to $b$ it is good for $A$ with
respect to any $b'$ close enough to $b$.

For each $a_i\in A$ let $B_i\ni a$ be an open box such that $B_i
\subseteq G(b)$ if $a_i\in G(b)$ and $B_i\cap G(b) =\0$ otherwise.
Let $l'$ be a line parallel to $l$ such that $l'\cap B_i \neq \0$
for all $i$. Choosing $l'$ close enough to $l$ we may assume that
$G(b)$ induces a maximal non-constant partition of $l'$ (as usual,
among all generic $b'$ such that $(b',l)\in \mathcal B$). Choose
$a'_i\in B_i \cap l'$ witnessing this and denote this set of
points $A'$. It will be convenient to take $A'$ to be the
projection of $A$ to $l'$. 

We slightly change our definition of $\pot$. Instead of requiring
that $x_1\pot x_2$ if and only if the set of $y\in Z$ such that
$y\po x_1 \wedge y\not\po x_2$ is finite (having $Z$ defined as
above) we replace $Z$ by $Z'$ where $Z':=\{z\in Z\mid\, \models
R_{b'}(z,A')\}$. I.e. we restrict ourselves to the set $Z'$ of
$b'$ such that $G(b')$ induces a non-constant maximal partition on
both $l$ and $l'$ and $b'$ is good for both $l$ and $l'$ (with
respect to $b$ in both cases). Because $b$ is generic over all the
data, $Z'$ contains a small open neighbourhood of $b$ which must be
a $2$-dimensional set ($n$-dimensional in the general case). Let
$x\in
\partial G(b)$ be generic over all the data (so in particular not
on $l,l'$) such that $x$ lies between $l$ and $l'$ (see Figure 2).

\begin{center}
\includegraphics[height=2in,width=4in]
{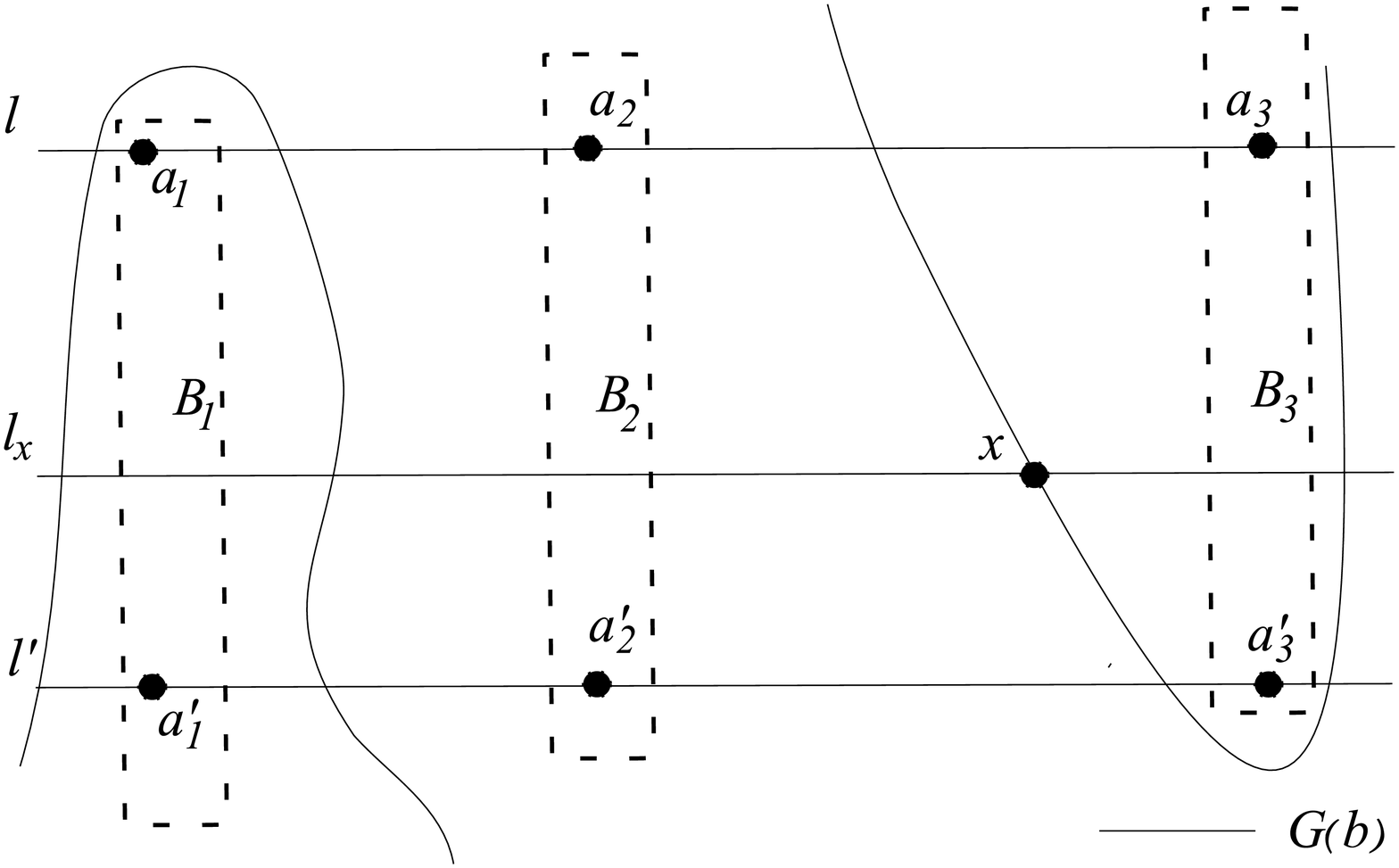}\\
Figure 2
\end{center}

\noindent Let $l_x$ be the line through $x$
parallel to $l,l'$. We would like to show that
\begin{equation}
\{y\in l_x\mid x \pot y\} \mbox{\,\,has $x$ in its
boundary}.\tag{$**$}
\end{equation}
This will be enough ($b$ was taken to be generic and $x$ is a generic point
in $\partial G(b)$, so moving $b$ we will get a 2-dimensional set
of $x$ with the required property). In the general $n$-dimensional
case we would have to repeat the same process several times, but
the main idea is unaltered.

We do not claim that this will be outright true, but we will now
start a process which will provide the desired result. Throughout,
when working within a line (parallel to one of the axis) we will
use the natural induced order.

Assume $(**)$ is not true. By definition $x\in \partial G(b)$ and
for concreteness assume that for all $x < y$ in $l_x$ there is
$z\in l_x$ with $x< z < y$ such that $z\in G(b)$ (as in the figure above). Since ($**$) is
assumed not to hold, for all $y > x$ there exists some $z\in l_x$ between $x$ and
$y$ such that $z\notin G_t(x)$. By the definition of $\pot$ this
means that there are infinitely many $b'\in Z'$ such that $x\in
G(b')$ but $z\notin G(b')$. Fix any such $\CM$-generic $b'$.

Let $(a_i,a_i')$ be the line segment between $a_i$ and $a_i'$.
Observe that if $a_i\in G(b)$ then by choice of $l'$ and $a_i'$
the segment $(a_i, a'_i)$ is contained in $B_i$ (and therefore in
$G(b)$). Similarly, $(a_i,a'_i)\cap G(b) = \0$ if $a_i\not\in
G(b)$. Since $b'$ was chosen generic over all the data, we also
know that $G(b')$ contains open neighbourhoods of $(a_i,a'_i)$ or
open neighbourhoods disjoint from $G(b')$.

Assume first that for all $i$ we have $(a_i,a_i')\subseteq G(b')$
whenever $a_i\in G(b)$ and $(a_i,a_i')\cap G(b')=\0$ otherwise.
This implies that $l_x$ has more sign changes (witnessed by
$G(b')$) than $l$ did (with respect to $b$) and it is easy to
verify that $b'$ (and $z$) can be chosen so that the partition is
not constant. But $l$ was chosen so that the size of a maximal
non-constant partition (with respect to a generic element $b$) is
maximal, so we would have a contradiction.

So we may assume that for some $i$ either $a_i\in G(b)$ and
$(a_i,a_i')\not\subseteq G(b')$ or $(a_i,a_i')\cap G(b')\neq \0$
for some $a_i\not\in G(b)$. For concreteness we will assume the
former holds (see Figure 3) and let $a''\in (a_i, a_i')$ be such that
$a_i''\not\in G(b')$ be generic over all the data.

We will now change our definition of $\pot$ once more replacing
$Z'$ with $Z'':=\{z\in Z'\mid a''\notin G(z)\}$. Take $l''$ close
enough to $l$ so that $G(b')$ induces on $l''$ a maximal
non-constant partition. Now restart the whole process with
$l,l''$, $G(b')$ and $Z''$. By o-minimality, this process of
restricting $Z$ to obtain more intervals in $G(b)\cap (a_i,a_i')$
cannot go on forever. So after finitely many such changes we are
reduced to the case where $(a_i,a_i')\subseteq G(b')$
if $a_i\in G(b)$ and $(a_i,a_i')\cap G(b') = \0$ otherwise, which
we already proved contradicted the fact that $x$ was a non
isolated point in $L_t(x)\cap l_x$. This concludes the proof of the theorem.

\begin{center}
\includegraphics[height=2in,width=4in]
{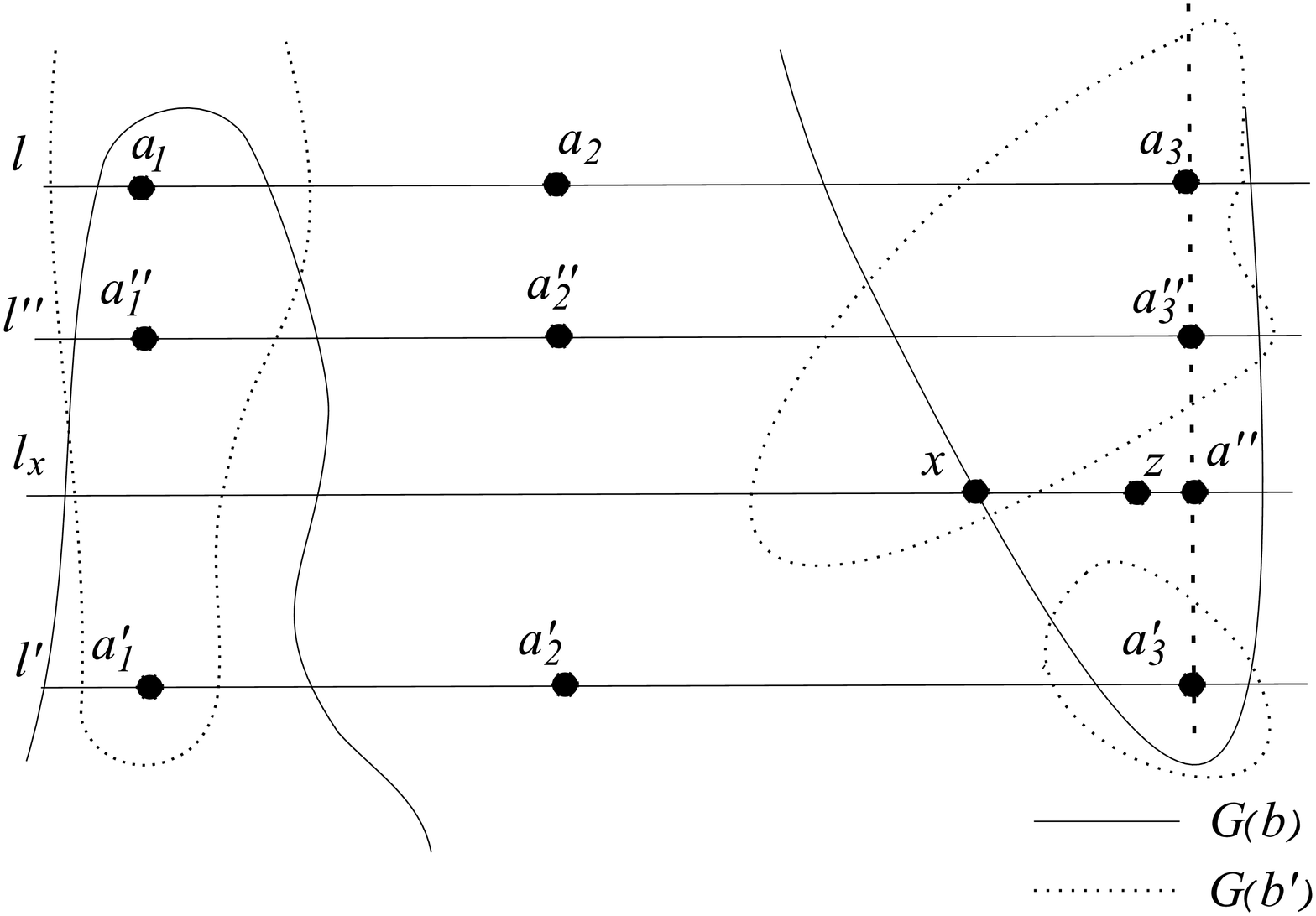}\\
Figure 3
\end{center}

\bigskip

\noindent The results obtained thus far imply by induction that we can define
a finite by o-minimal set in any purely unstable structure (see definition \ref{punstable}) interpretable in an o-minimal theory, proving the first part of Theorem \ref{mainintro}. This theorem cannot be outright strengthened as the following example demonstrates. 
Consider the structure $\CN:= (\mathbb R^2,\prec)$
where $(x_1,x_2) \prec (y_1,y_2)$ is interpreted as $x_1 < y_1$.
So $\CN$ is unstable of o-minimal dimension 2. It is an easy
exercise to check that any infinite definable subset of $N$ is
either 2-dimensional or stable, and that the only way to obtain an
o-minimal structure in $\CN$ is to work in $\CN^{eq}$. 

The above example is not the only obstacle on the way of completing the proof of Theorem \ref{interpreting}. Towards that end we will also need to improve the results of this section in order to find an \emph{unstable} $\CN$-definable set of
small $\CM$-dimension. Before proceeding to this task, we conclude the present section with a discussion of the second part of Theorem \ref{mainintro}: 

\begin{corollary}\label{localmain}
 Let $\CN$ be interpretable in an o-minimal $\kappa$-saturated structure $\CM$ and $p\in S^{\CN}(N_0)$ for some $N_0\prec N$ with $|N_0|<\kappa$. Then there exists a non-algebraic $p\subseteq q\in S^{\CN}(N)$ which is either strongly stable or finite by o-minimal. 
\end{corollary}

Since this result will not be used below we only give a sketch of the proof. If $p$ is strongly stable every non-algebraic $q\supseteq p$ is strongly stable and we have nothing to do. So we may assume that $p$ is weakly unstable. by extending $N_0$ if needed (preserving the cardinality) we may assume that $p$ is unstable, and by saturation we may also assume that there is an $\CN$-definable partial order $\po$ with infinite chains in $p$. If $\mdim p = 1$ the result follows from Theorem \ref{theorem}. So we may assume that $\mdim p = n > 1$. To simplify things we will assume that $\po$ is $\0$-definable in $\CN$. For every $\phi\in p$ denote $N_{\phi} := \{a\mid a \models \phi(x)\}$ and $\CN_{\phi}$ the $\po$-structure whose universe is $N_{\phi}$ and where $\po$ is interpreted as its restriction therein. Clearly, $\CN_{\phi}$ is unstable. 

Our goal is to show that there exists a $\po$-formula $\psi(x,y)$ over $\0$ such that for all $\phi\in p$ there exists $c\subseteq N_{\phi}$ satisfying $0 < \dim (\psi(x,c) \land \phi(x)) < n$. By compactness (and the definability of o-minimal dimension) this would imply the existence of some $c$ such that $0 < \dim (\psi(x,c) \land \phi(x)) < n$ for all $\phi \in p$. In particular, we may assume that $0 < \dim \psi(x,c) < n$. By assumption $\psi(x,c) \cup p$ is infinite, with the desired conclusion following by induction. 

The starting point of the proof is the observation that, assuming Lemma \ref{close section 2} and Proposition \ref{random crap}, the proof of Theorem \ref{dimension down} assures that there are $a,b$ such that $0 < \dim (a,b)_{\pot} < n$. Note also that the formula defining $\pot$ (given in ($\dagger$) above) depends on $\CN$ only by the use of parameters and, more significantly for us, by determining the maximal possible size of a non-constant partition of a line through $\CN$. We leave it as an exercise to verify that by saturation of $\CM$, working in an infinitesimal (with respect to $N_0$) neighbourhood of some (generic) $e\models p$ we can find a line $l$ so small that the maximal size of a partition of $l$ in $\CN_{\phi}$ is uniformly bounded (i.e. does not depend on $\phi$). Thus, under the assumption that Lemma \ref{close section 2} and Proposition \ref{random crap} hold in any $\CN_{\phi}$, the plan described in the previous paragraph can be carried out. 

So we only have to take care of the case that one of the above propositions does not hold in cofinitely many of the $\CN_{\phi}$. If Proposition \ref{random crap} is the one that fails then, inspecting its proof, we get that there are $a,b$ such that $0 < \dim ((a,b)_{\po} \cap \phi(\CM)) < n$ with the desired conclusion. If Lemma \ref{random crap} is the one causing problems we need to note that the lemma is only used to show that $\pot$ (which, as we have observed, can be defined uniformly) is generically closed. If the proof of that statement fails, there must be some infinite set of the form $Z\cap (L(c)\setminus L(a))\cap N_\phi$ which is of small dimension, for the definable set $Z$ appearing in $(\dagger)$. Since $Z$ does not depend on $\phi$ the conclusion follows. 

\section{Interpreting an o-minimal structure}\label{interpretingsec}

In this section we complete the induction introduced in the previous section to conclude that any unstable theory
interpretable in an o-minimal structure interprets itself an
o-minimal structure. As the example preceding Corollary \ref{localmain} illustrates, in order to achieve this we cannot avoid
working in $\CN^{eq}$. Although it seems plausible that our
argument could be carried out to $\CM^{eq}$ this would require
some additional technical tools. In order to avoid such
technicalities we will assume from now on that $\CM$ eliminates
imaginaries.

We will work with both stable formulas and definable stable
sets. To prevent any confusion we will use ``stable'' for
definable stable sets and refer to stable formulas as formulas
satisfying NOP.

\subsection{Preliminaries}

We need the some definitions and results from \cite{onshuus}
and \cite{onshuusthesis}.

\begin{definition}\label{thorking}
A formula $\delta (x,a)$ \emph{strongly divides over $A$} if
$\tp(a/A)$ is non-algebraic and $\{ \delta (x,a') \}_{a'\models
\tp(a/A)}$ is $k$-inconsistent for some $k\in \mathbb{N}$.
 $\delta (x,a)$ \emph{\th-divides over $A$} if we can
find some tuple $c$ such that $\delta (x,a)$ strongly divides over
$Ac$.

\end{definition}

Standard forking terminology generalises naturally to
\th-forking. For example, a formula \th-forks over a
set $A$ if it implies a finite disjunction of formulas 
\th-dividing over $A$. In particular, $\uth$-rank is
the foundation rank of the partial order (defined on complete
types) $p<_\tho q$ defines as ``$p$ is a \th-forking extension of $q$''.

The \th-rank of a formula is the analogue of the global rank in
simple theories. That is, \th$(\varphi(x,b))\geq \alpha+1$ if
there is $\psi(x, c) \vdash \varphi(x, b)$ \th-dividing over $b$
with \th$(\psi(x,c))\geq \alpha$.

\begin{fact}
Let $\CN$ be definable in an o-minimal structure
$\CM$, let $\phi(x,b)$ be $\CN$-definable and let $p(x)\in S_n^{\CN}(N)$. Then both $\tho (\phi(x,b))$ and $\uth(p(x))$ are
finite.
\end{fact}

\begin{proof}
Any instance of \th-forking in $\CN$ implies an instance of
\th-forking in $\CM$ so the \th-rank of any $\CN$-definable set
will be bounded by the dimension of the corresponding set in
$\CM$. This gives a finite bound for the global \th-rank for
structures interpretable in o-minimal theories which depends only
on the arity of the variable $x$.

As in superstable theories (see \cite{pillaybook}), if we define
\[
\tho(p(x)):=\min\left\{ \tho\left(\phi\left(x,b\right)\right)
\left|\right. \phi\left(x,b\right)\in p\left(x\right)\right\}
\]
then

\[
\uth(p(x))\leq \tho(p(x))
\]
which completes the proof.
\end{proof}
\medskip

\noindent The following is Theorem 5.1.1 in \cite{onshuus}:

\begin{fact}\label{thesis}
If $\phi(x,y)$ satisfies NOP and there is a $\phi$-formula
witnessing that $\tp(a/bc)$ forks over $c$, then there is a
$\phi$-formula witnessing that $\tp(a/bc)$ \th-forks over $c$.
\end{fact}
\noindent and 
\begin{fact}\label{withkobi}
If $T$ is dependent and  $\phi(x,b)$ is a definable stable set then for every formula
$\psi(x,y)$ the formula $\phi(x,b)\wedge \psi(x,y)$ satisfies NOP.
\end{fact}
\noindent is well known, see for example \cite{onshuus-peterzil}. Gathering all of the above, we obtain: 

\begin{corollary}\label{superstable}
In a dependent theory, if a type $p$ contains a formula defining a
stable set then $\uth(p)=\ur(p)$. In particular, if $\uth(p)$ is
finite then $p$ has finite $\ur$-rank.
\end{corollary}

\begin{proof}
Any instance of \th-forking is an instance of forking so

\[\ur(p)\geq \uth(p)\] for any type $p$. We prove the other
inequality by induction. For $\alpha =0$, $\ur(p)\geq 0$ if and only if
$p$ is consistent if and only if $\uth(p)\geq 0$. If $\ur(p)\geq
\alpha$ for $\alpha$ limit, the claim follows from the induction
hypothesis and the fact that both $\uth(p)\geq \alpha$ and
$\ur(p)\geq \alpha$ if and only if $\uth(p)\ge \delta$ (resp.
$\ur(p)\ge \delta$) for all $\delta < \alpha$.

It remains to deal with $\ur(p)\ge \alpha +1$. Assume inductively 
that for any type $q$ containing a formula defining a stable set,
if $\ur(q)\geq \alpha$ then $\uth(q)\geq \alpha$. Now let $p\in
S(A)$ contain a formula $\phi(x,a)$ defining a stable set, and
assume that $\ur(p)\geq \alpha+1$.

By definition there is some $r\supseteq p$  such that
$\ur(r)\geq \alpha$ and $r$ forks over $A$ witnessed by some
formula $\theta(x,b)$. Let $\theta'(x,y):=\theta(x,y)\wedge
\phi(x,a)$; $\theta'(x,y)$ satisfies NOP by Fact \ref{withkobi}
and clearly $\theta'(x,b)$ forks over $A$. By Fact \ref{thesis}
there is a $\theta'$-formula which witnesses that $r$ \th-forks
over $A$, by induction $\uth(r)\geq \alpha$, and by definition
$\uth(p)\geq \alpha+1$.
\end{proof}

\begin{fact}\label{algebraic}
Let $\phi(x_1,x_2)$ be such that $\models \forall x_1\exists^{\leq
n} x_2 \phi(x_1,x_2)$ for some $n\in \mathbb{N}$, and let
$\phi'(x):=\exists y \phi(x,y)$. Then the following hold.

\begin{itemize}

\item $\phi(\mathcal{C}^2)$ is a stable definable set if and only
if $\phi'(C)$ is a stable definable set.

\item $\Tho (\phi(\mathcal{C}^2))=\Tho (\phi'(C))$.
\end{itemize}
\end{fact}

\begin{proof}

If $\phi'$ is unstable, then any formula witnessing NOP for
$\phi'$ will also witness NOP for $\phi$. Thus, it's enough to
show that if $\phi$ is unstable so is $\phi'$. By \ref{shelah}
there is an indiscernible sequence $\bar {\bf b}$ satisfying
$\phi$ and $\pi(x_1,y_1;x_2,y_2)$ witnessing the strict order
property. Let $b_i:=(a_i,c_i)$ be the i-th element of $\bar {\bf
b}$; by assumption $c_i\in \acl(a_i)$. Therefore,
$\tp(c_1,c_2/a_1,a_2)$ is isolated, say by
$\psi(y_1,y_2,a_1,a_2)$. Hence, the formula $(\forall
x_1,x_2)(\psi(y_1,y_2,x_1,x_2)\to \pi(x_1,y_1;x_2,y_2))$ orders
the sequence $\bar {\bf a}:= \{a_i\}_i$. This proves (i).

The proof of (ii) is a straightforward induction on the \th-rank
of $\phi'(\mathcal{C})$.
\end{proof}

\begin{proposition}\label{non algebraic}
Let $\delta(x,y)$ and $\pi(y)$ be such that $\models \forall y
\left(\pi(y)\to \exists^\infty x \delta(x,y)\right)$ and $\Tho \left(\pi\left(y\right)\right)$ is finite. Then

\[\Tho \left(\delta\left(x,y\right)\wedge
\pi\left(y\right)\right)> \Tho \left(\pi\left(y\right)\right).\]
\end{proposition}

\begin{proof}
Clearly $\Tho \left(\delta\left(x,y\right)\wedge
\pi\left(y\right)\right)\geq \Tho \left(\pi\left(y\right)\right)$.
We will prove the sharp inequality by induction. The case
$\alpha=0$ being clear we assume that $\delta(x,y)$ and $\pi(y)$
are formulas over some set $A$ and suppose $\Tho
\left(\pi\left(y\right)\right)\geq \alpha+1$. By definition there
is a formula $\theta(y,b)$ \th-forking over $A$, such that
$\theta(y,b)\vdash \pi(y)$ and $\Tho(\theta(x,b))\ge \alpha$. As
$\models \forall y (\theta(y,b)\to \exists^\infty x \delta(x,y))$
the induction hypothesis yields
\[
\Tho(\theta(y,b)\wedge \delta(x,y))>\Tho(\theta(y,b)) = \alpha.
\]
But $\theta(y,b)\land \delta(x,y)$ \th-forks over $A$ (because
$\theta(y,b)\to \exists x \delta(x,y)$) so the claim follows.
\end{proof}

\begin{proposition}\label{strong dividing}
Assume that $\tp(a/Ab)$ \th-forks over $A$ and $\tp(a/Ab)$ has
finite $\ur$-rank. Then there are $b',c$ such that
$\ur(\tp(a/Ab))=\ur(\tp(a/Abb'c))$ and $\tp(a/Ab'c)$ strongly
divides over $Ac$.
\end{proposition}

\begin{proof}
By definition, there are finitely many formulas $\phi_i(x,b_i)$
such that

\[ \tp(a/Ab)\vdash \bigvee_i \phi_i(x,b_i) \] and $\phi(x,b_i)$
\th-divides over $A$. Finiteness of the $\ur$-rank implies that
$\tp(a/Ab)$ does not fork over $Ab$ and therefore $\ur(\tp(a/Ab)\cup \{\phi_m(x,b_m)\})=\ur(\tp(a/Ab))$
for some $m$. For such $m$, we get that $\tp(a/Ab)\cup \{\phi_m(x,b_m)\}$ is a non
forking extension of $\tp(a/Ab)$; using automorphisms we may
assume that $a\models \phi_m(x,b_m)$ and $a\ind_{Ab} b_m$.

By definition of \th-dividing there is some $c'$ such that
$\phi(x,b_m)$ strongly divides over $Ac'$. Let $c\models
\tp(c'/Abb_m)$ be such that $a\ind_{Abb_m} c$. Since $c\models
\tp(c'/Ab_m)$ strong dividing is preserved and

\[\ur\left(\tp\left(a/Ab\right)\right)=
\ur\left(\tp\left(a/Abb_m\right)\right)=\ur\left(\tp\left(a/Abb_mc\right)\right)\]
so letting $b'=b_m$ proves the proposition.
\end{proof}

\begin{proposition}\label{stable formula}
If a set $\phi(x,b)$ is stable, then there is some $\theta(y)\in
\tp(b/\emptyset)$ such that $\phi(x,y)\wedge \theta(y)$ has NOP.
\end{proposition}

\begin{proof}

\begin{claim}\label{claim}
Let $r(y)=\tp(b/\emptyset)$. Then there are no indiscernible
sequences $\langle a_i\rangle _{i\in \omega}$ and $\langle
b_j\rangle _{j\in \omega}$ such that $b_j\models r(y)$ and
$\models \phi(a_i, b_j)$ if and only if $i\leq j$.
\end{claim}

\begin{proof}
We may assume without loss of generality (by using automorphisms)
that $b_0=b$. In this case, $a_i\models \phi(x,b)$ which implies
that $\phi(x,y)\wedge \phi(x,b)$ witnesses the order property,
contradicting Fact \ref{withkobi}.
\end{proof}

Let $p(x_1,\dots, x_n, \dots)$, $q(y_1,\dots, y_n, \dots)$ be the
(partial) types expressing ``$\langle x_i\rangle _{i\in \omega}$
is an indiscernible sequence'' and ``$\langle y_j\rangle _{j\in
\omega}$ is an indiscernible sequence'' respectively. By Claim
\ref{claim} the type

\[p\left(\bar{x}\right)\cup q\left(\bar{y}\right)\cup
\left\{ \bigcup_j r\left(y_j\right)\right\} \cup
\left\{\bigcup_{i<j} \phi\left(x_i, y_j\right)\right\} \cup
\left\{\bigcup_{i\geq j} \neg\phi\left(x_i, y_j\right)\right\}\]
is inconsistent. By compactness, there is a formula $\theta(y)\in
r(y)$ such that

\[p\left(\bar{x}\right)\cup q\left(\bar{y}\right)\cup
\left\{ \bigcup_j \theta\left(y_j\right)\right\} \cup
\left\{\bigcup_{i<j} \phi\left(x_i, y_j\right)\right\} \cup
\left\{\bigcup_{i\geq j} \neg\phi\left(x_i, y_j\right)\right\}\]
is inconsistent and by definition this implies that
$\phi(x,y)\wedge \theta(y)$ has NOP.
\end{proof}

\begin{corollary}\label{corollary}
Let $a,b$ and $c$ be such that there exists $\phi(x,b)\in
\tp(a/bc)$ defining a stable set but $\ur(a/c)=\infty$. Then
$a\nthind_c b$.
\end{corollary}

\begin{proof}
Let $a,b,c$ and $\phi(x,b)$ be as in the statement. By Proposition
\ref{stable formula} there is some $\theta(y)\in \tp(b)$ such that
$\psi(x,y):=\phi(x,y)\wedge \theta(y)$ has NOP.

By definition every type extending $\phi(x,b)$ has finite $\ur$-rank and $\tp(a/c)$ has infinite $\ur$-rank so every completion of
$\tp(a/c)\cup \{\phi(x,b)\}$ forks over $c$. This implies that
$\psi(x,b)$ forks over $c$. But $\psi(x,y)$ satisfies NOP so Fact
\ref{thesis} implies there is a $\psi$-formula witnessing that
$\tp(a/bc)$ \th-forks over $c$.
\end{proof}

\subsection{Unstable sets of small dimension}

In this section we conclude the proof of Theorem \ref{interpreting}. The main result which allows us to do this is:

\begin{theorem}\label{unstable}
let $T$ be a dependent theory with definable finiteness (i.e. $T$
eliminates the quantifier $\exists^{\infty}$) such that the
(global) \th-ranks of definable sets are bounded by a finite
number. Let $\mathbf{\Phi}$ be the class of unstable definable sets
in $T$, and
\[
 n:=\min \left\{ \text{\th}\left(\phi(\mathcal{C},c)\right) \left| \right. \ \phi(x,c)\in
\mathbf{\Phi} \right\}.
\]
Let $\phi(x,c)\in \mathbf{\Phi}$ be such that \th$(\phi(x,c))=n$
and let $\theta(x,b)\vdash \phi(x,c)$ contain some $a$ with $\ur(\tp(a/c))=\infty$ and $a\not\in
\acl(b)$. Then $\theta(x,b)$ is unstable.
\end{theorem}

\begin{proof}
Let $\phi(x,c)$ be an unstable set of minimal $\Tho$-rank. Assume towards a contradiction that there is a stable set
$\theta(x,b)$ and $a\models \theta(x,b)$ satisfying all the assumptions of the theorem. Fix such $\theta(x,b)$ and $a$ for which
$m:=\ur(\tp(a/cb))$ is maximal. Note that $m$ is well defined as $\ur(\tp(a/cb))$, when finite, is
bounded by $\uth(a/\0)$.
To simplify the notation we will assume  that $c$ is a subsequence
of $b$. 

\medskip

To reach a contradiction we will use \th-forking to find a
definable subset $Y$ of $\phi(x,c)$ consisting of an infinite definable
family of almost disjoint copies of $\theta(x,b)$. We can then show that either $I$, the set parametrising the family, is stable which implies that so
is $Y$ - contradicting the maximality of $\ur(\tp(a/b))$; or $I$ is
unstable with $\tho(I)<n$, in contradiction to the
minimality of $n$.

\medskip

By Corollary \ref{corollary} we know that $a\nthind_c b$ and by
hypothesis $\ur(\tp(a/b))=m$. By Proposition
\ref{strong dividing} there are $b_0,d$ such that $a\ind_{b} b_0d$
and $\tp(a/b_0cd)$ contains some $\theta_0(x,b_0)$ 
strongly dividing over $cd$. By definition of strong dividing
$\tp(b_0/cd)$ is non algebraic and there is a formula $\pi(y,d)$
such that

\[\left\{ \theta_0\left(x,b'\right) \right\}_{b'\models
\pi(y,d)}\] is $k$-inconsistent for some $k$. In particular, there
are at most $k-1$ elements satisfying $\pi(y,d)\land\theta_0(a,y)$
so $b_0$ is algebraic over $ad$, witnessed by the formula
$\theta_0(a,y)\wedge \pi(y,d)$.

By assumption $\tp(a/bc)$ is non algebraic and $a\ind_{bc} b_0d$
so $\tp(a/b_0d)$ is non algebraic and by definable finiteness we
know there is a formula $\mu(y)\in \tp(b_0/cd)$ such that
$\phi(x,b')$ is infinite for any $b'\models \mu(y)$; we may assume
that $\pi(y,d)\rightarrow \mu(y)$.

\begin{claim}
$\tp(a/cd)$ contains no stable definable sets.
\end{claim}

\begin{proof}
Assume that $\delta(x)\in \tp(a/cd)$ defines a stable set. Since
$b_0\in \acl(ad)$ there is a formula $\delta(x,y)\in \tp(ab_0/cd)$
such that $\forall x \exists^{<n}y \delta(x,y)$. We may also
assume that $\exists y \delta(x,y) \equiv \delta(x)$ so, by Fact
\ref{algebraic}, $\delta(x,y)$ defines a stable set as well.

Since $\ur(a/cd)\ge \ur(a/bd)$ (recall that $c\subseteq b$) the maximality of $m$ implies that in fact $\ur(a/cd)= \ur(a/bd)$. But $b_0\in \acl(ad)$ so $\ur(ab_0/cd)=\ur(a/cd)=\ur(a/bd)$. But $\ur(a/bdb_0)=\ur(a/b)=m$, so $\ur(a/cd)=m$. On the other hand, as $\delta(x,y)$ defines a stable set, we can use Lascar's inequalities to get: 
\[
 \ur(ab_0/cd)=\ur(b_0/cd)+\ur(a/bdb_0)
\]
But $b_0\notin \acl(cd)$ and the last equality implies that $\ur(a/cd)>m$, a contradiction. 

\end{proof}

By Fact \ref{algebraic} we know that $\tp(ab_0/cd)$ does not
contain any non algebraic formula defining a stable set. In
particular, $\pi(\mathcal{C},d)$ is unstable.

\begin{claim}
$\Tho (\pi(y,d))<\Tho (\phi(x,c))=n$.
\end{claim}

\begin{proof}
Let $\psi(x,y;c,d):=\phi(x,c)\wedge \theta_0(x,y)\wedge \pi(y,d)$.
For all $a'\models \phi(x,c)$ there are finitely many $b'$ such
that $\models \psi(a',b';c,d)$; by Fact \ref{algebraic}, $\Tho
(\psi(x,y;c,d))=\Tho(\phi(x,c))=n$.

However, $\theta_0(x,b')\wedge \phi(x,c)$ is non algebraic for any
$b'\models \pi(y,d)$ so by Fact \ref{non algebraic} $\Tho
(\psi(x,y;c,d))>\Tho (\pi(y,d))$.
\end{proof}

So $\pi(\mathcal{C},d)$ is an unstable definable set of \th-rank
smaller than $n$; this contradicts the minimality of $n$ and the
theorem follows.
\end{proof}

\noindent We can now prove Theorem \ref{interpreting}:

\begin{corollary}
Let $\CM:=(M,<,\dots)$ an o-minimal structure with
elimination of imaginaries and a dense underlying order. Let $\CN$ be unstable
interpretable in $\CM$. Then $\CN$ interprets an o-minimal structure.
\end{corollary}

\begin{proof}
Let $\Phi$ be the set of all (non algebraic) unstable
$\CN$-interpretable sets. For each $Z\in \Phi$ let $(\Tho(Z),
d(Z))$ be the pair consisting of the \th-rank and the o-minimal
dimension of $Z$. Let $Y\in \Phi$ minimise $(\Tho(Y), d(Y))$ in
the lexicographic order.

\begin{claim}\label{ominimal}
The o-minimal dimension of $Y$ is $1$.
\end{claim}

\begin{proof}
Suppose otherwise. Because $\CM$ eliminates imaginaries $Y$ is
definable in $\CM$. Let $A$ be a set over which $Y$ is definable.
Because $Y$ is unstable Theorem \ref{dimension down}
implies that there is some $Y_0\subset Y$ $\CN$-definable over $B\supset
A$ with $\dim Y_0 < \dim Y$. Moreover, there exists $a\in Y_0\setminus \acl(B)$ such that $\ur(\tp(a/A))=\infty$. By Theorem
\ref{unstable} $Y_0$ is unstable so by definition $Y_0\in \Phi$
contradicting the minimality of $(\Tho(Y), d(Y))$.
\end{proof}

By Claim \ref{ominimal} there is some $Y\in \Phi$ such that the
o-minimal dimension of $Y$ is 1. By Corollary \ref{last corollary}
$\Th(Y)$ interprets an o-minimal structure, with the desired conclusion.
\end{proof}
\noindent Recall the following from \cite{onshuus-peterzil}:
\begin{definition}\label{punstable}
A definable set $\phi(x,a)$ is \emph{purely unstable} if every
definable subset of $\phi(x,a)$ is unstable.
\end{definition}
\noindent So the last corollary shows, in particular: 
\begin{corollary}
 Let $\CM:=(M,<,\dots)$ be a dense o-minimal with elimination of imaginaries and $\CN$ unstable interpretable in $\CM$. Then $\CN$ interprets a purely unstable set. 
\end{corollary}

Unfortunately, unlike the results of the previous sections, the
present proof does not seem to give significant local data. This
is one of the reasons why it is not clear to us, at this
stage, what should the right classification of theories
interpretable in o-minimal structures look like. The great
flexibility in creating local phenomena in o-minimal structures
(and to some extent even more so in their reducts) suggests that
analysability of types with respect to some ``nice'' collection of
types should be the right direction and the results of this paper
suggest that the class of o-minimal (by finite) types has
a crucial role in any such analysis. A sharpening of the results
of the present section could provide some level of analysis in
such terms for weakly unstable types, but the situation in the
stable part of the picture is much less obvious.

For types that contain a definable stable set an analysis exists
in terms of regular types, hopefully satisfying Zilber's
Trichotomy. Such a classification would give a good solution for the class of stably dominated types. However, as the example in Remark \ref{wstable} shows, not all
stable types (not even all strongly stable types) contain a
formula defining a stable set, or are even stably dominated. The following questions seem
natural, and will probably require some additional work:

\begin{enumerate}

\item Is there a (natural) geometric stability theoretic
distinction between stable and unstable types (aside from Shelah's
combinatorial definition) and between stable and strongly stable
types in reducts of o-minimal theories.

\item What role do stable regular types play in the space of types
of a reduct of an o-minimal theory. Do they satisfy Zilber's
trichotomy. Can a reasonable theory of analysability be developed
in reducts of o-minimal theories in terms of regular types and
finite by o-minimal types.

\item Since our local results relate only to types over models, it
seems natural to ask whether, in the present context, a reasonable
notion of prime models (over arbitrary sets) - an equivalent of
a-models in stable theories - exists and what is the right
framework for the development of such a theory.
\end{enumerate}

As we pointed out in the introduction, the results of \cite{Sh900}
- though not sufficient if one hopes for as sharp results as we
would like to obtain - suggest several directions of research that
may be of relevance to the above question.

\bibliographystyle{alpha}

\bibliography{all}

\end{document}